\documentclass{article}

\usepackage{amssymb,amsmath,latexsym,amsthm,eucal}

\newfont{\Bb}{msbm10 scaled\magstep0}
\newfont{\Bbl}{msbm10 scaled\magstep1}
\newfont{\Bbs}{msbm10 scaled 800}

\newcommand{\fq}{\mbox{\Bb{F}}_q}
\newcommand{\f}{\mbox{\Bb{F}}}
\newcommand{\fp}{\mbox{\Bb{F}}_p}
\newcommand{\sfq}{\mbox{\Bbs{F}}_q}

\newcommand{\sff}{\mbox{\Bbs{F}}}
\newcommand{\fqk}{\mbox{\Bb{F}}_{q^k}}
\newcommand{\Z}{\mbox{\Bb{Z}}}

\newcommand{\Q}{\mbox{\Bb{Q}}}
\newcommand{\Qp}{\mbox{\Bb{Q}}_{p}}

\newcommand{\sZ}{\mbox{\Bbs{Z}}}
\newcommand{\C}{\mbox{\Bb{C}}}
\newcommand{\sC}{\mbox{\Bbs{C}}}

\newcommand{\LL}{L} 
\newcommand{\OL}{{\mathcal O}_{\LL}}
\newcommand{\AOL}{A_{\OL}}

\newcommand{\Robba}{{\mathcal R}}

\newcommand{\Hyper}{\mbox{\Bb{H}}}
\newcommand{\ord}{{\rm ord}}

\newcommand{\Spec}{{\rm Spec}}
\newcommand{\Ker}{{\rm Ker}}
\newcommand{\Coker}{{\rm Coker}}
\newcommand{\Ima}{{\rm Im}}
\newcommand{\lcm}{{\rm lcm}}
\newcommand{\Gal}{{\rm Gal}}
\newcommand{\GL}{{\rm GL}}
\newcommand{\HHom}{{\rm Hom}}
\newcommand{\Basis}{{\mathcal B}}

\newcommand{\W}{W}
\newcommand{\K}{K}

\newcommand{\F}{F} 
\newcommand{\Kalg}{L}

\newcommand{\Xp}{X}
\newcommand{\Sp}{S}
\newcommand{\Yp}{Y}
\newcommand{\Tp}{T}

\newcommand{\Pp}{P}

\newcommand{\XK}{\mathcal{X}_{K}}
\newcommand{\SK}{\mathcal{S}_{K}}

\newcommand{\XW}{{\mathcal X}}
\newcommand{\SW}{{\mathcal S}}

\newcommand{\XL}{\mathcal{X}_{\LL}}
\newcommand{\SL}{\mathcal{S}_{\LL}}
\newcommand{\XOL}{\mathcal{X}_{\OL}}
\newcommand{\SOL}{\mathcal{S}_{\OL}}

\newcommand{\XC}{X}
\newcommand{\SC}{S}
\newcommand{\YC}{Y}

\newcommand{\XF}{X}
\newcommand{\SF}{S}
\newcommand{\YF}{Y} 
\newcommand{\UF}{U} 

\newcommand{\PF}{P}

\newcommand{\UnivR}{R} 

\newcommand{\XR}{X}
\newcommand{\SR}{T}
\newcommand{\UR}{U}

\newcommand{\PR}{P}

\newcommand{\Fr}{{\mathcal F}} 
\newcommand{\N}{{\mathcal N}} 
\newcommand{\HH}{{\mathcal H}} 

\newcommand{\AW}{A} 
\newcommand{\AK}{A_{\K}}
\newcommand{\Adag}{A^\dagger}
\newcommand{\AdagK}{A^\dagger_{\K}}
\newcommand{\AL}{A_{\LL}}

\newcommand{\AF}{A}

\newcommand{\pr}{\mbox{\Bb{P}}}
\newcommand{\spr}{\mbox{\Bbs{P}}}
\newcommand{\Aff}{\mbox{\Bb{A}}}

\newcommand{\E}{\mathcal E}

\newtheorem{theorem}{Theorem}[section]
\newtheorem{lemma}[theorem]{Lemma}

\newtheorem{conjecture}[theorem]{Conjecture}
\newtheorem{corollary}[theorem]{Corollary}

\newtheorem{definition}[theorem]{Definition}

\newenvironment{exafont}{\begin{sc}}{\end{sc}}
\newenvironment{example}{\vspace{0.3cm}\par\noindent\refstepcounter{theorem}\begin{exafont}Example
 \thetheorem\end{exafont}\hspace{\labelsep}}{\vspace{0.3cm}\par}
\newenvironment{note}{\vspace{0.3cm}\par\noindent\refstepcounter{theorem}\begin{exafont}Note
 \thetheorem\end{exafont}\hspace{\labelsep}}{\vspace{0.3cm}\par}

\title{Degenerations and limit Frobenius \\ structures in rigid cohomology}

\author{Alan G.B. Lauder
\\ \\
{\small \it In loving memory of my father}
\\ {\small \it George Alan Lauder 1937-2008}}

\begin{document}

\maketitle

\begin{abstract}
We introduce a ``limiting Frobenius structure'' attached to any degeneration of projective varieties over a finite field
of characteristic $p$ which satisfies a $p$-adic lifting assumption. Our limiting Frobenius structure is shown to
be effectively computable in an appropriate sense for a degeneration of projective hypersurfaces. We conjecture
that the limiting Frobenius structure relates to the rigid cohomology of a semistable limit of the degeneration through
an analogue of the Clemens-Schmidt exact sequence. Our construction is illustrated, and conjecture supported, by
a selection of explicit examples.
\end{abstract}

\section{Introduction}\label{Sec-Intro}

This paper grew out of the author's attempt to generalise an algorithm for computing zeta functions of
smooth projective hypersurfaces over finite fields to the case of singular hypersurfaces. This existing algorithm 
is called the deformation method, as the geometric idea behind the method is to deform a smooth diagonal hypersurface
to the hypersurface in which one is interested. In this paper we degenerate smooth projective hypersurfaces to
singular projective hypersurfaces, possibly not even reduced or irreducible, and try to compute some 
$p$-adic cohomological data attached to the degeneration. The problem we address is, how does one
compute this cohomological data and what is its geometric meaning? Let us begin by sketching
the original deformation method for smooth projective hypersurfaces \cite{Gerk,Lfocm}.

Let $\fq$ be the finite field with $q$ elements of characteristic $p$, and
\[ \Pp_0 \in \fq[x_0,x_1,\cdots,x_{n+1}] \]
be a homogeneous polynomial of degree $d$ defining a smooth 
hypersurface $\Xp_0$ in $\pr^{n+1}_{\sfq}$. Recall that the zeta function of $\Xp_0$ is the formal
power series
\[ Z(\Xp_0,T) := \exp\left(\sum_{k = 1}^{\infty} |\Xp_0(\fqk)| \frac{T^k}{k} \right)\]
which encodes the number of points on $\Xp_0$ over the different finite extensions $\fqk$ of the
base field $\fq$.
By a famous theorem of Dwork this series is the local expansion at the origin of a rational function
with integer coefficients \cite{Dwork}.
Define
\[ \Pp_1 := x_0^d + x_1^d + \cdots + x_{n+1}^d \]
and assume $p$ does not divide $d$, so that the projective hypersurface $\Xp_1$ defined by
$\Pp_1$ is smooth. Let $\Xp \subset \pr^{n+1}_{\sfq} \times \Aff^1_{\sfq}$ be the subvariety of smooth
hypersurfaces in the pencil defined by the equation
\[ \Pp_t := (1-t) \Pp_0 + t \Pp_1 = 0. \]
Thus there is an open curve $\Sp \subseteq \Aff^1_{\sfq}$ and a smooth morphism $\Xp \rightarrow \Sp$ with
generic fibre the smooth hypersurface defined by $P_t$. One can construct to any $p$-adic
precision (in a well-defined sense) the relative rigid cohomology $\HH := H^n_{rig}(\Xp/\Sp)$ of the smooth projective family by starting with
the cohomology $H^n_{rig}(\Xp_1)$ of the smooth diagonal fibre and solving a $p$-adic differential equation.
Since $\Xp_0$ is also smooth a base change theorem allows one to specialise $\HH$ to an object
$\HH_0$ such that $\HH_0 = H^n_{rig}(\Xp_0)$. 
Let $\K$ be the unique unramified extension of the field of $p$-adic numbers $\Qp$ with residue
field $\fq$, and $\sigma: \K \rightarrow \K$ the Frobenius automorphism. 
We shall not say anything in detail here about rigid cohomology, except
that $H^n_{rig}(\Xp_0)$ is a vector space of finite dimension over $\K$
which is acted upon by a $\sigma$-linear map
$\Fr_0$ called Frobenius.
The trace formula in rigid cohomology
\begin{equation}\label{Eqn-TraceFormula}
 Z(\Xp_0,T) = 
\frac{\det(1 - T \Fr_0^{\log_p(q)}|H^n_{rig}(\Xp_0))^{(-1)^{n+1}}}{(1-T)(1-qT)(1- q^2T) \cdots (1-q^n T)}
\end{equation}
now allows us to recover the zeta function exactly, assuming we have carried enough $p$-adic precision in our computation.

When $\Xp_0$ is singular the above algorithm fails, both in theory and practice. One cannot specialise $\HH = H^m_{rig}(\Xp/\Sp)$
at the point $t = 0$ given the manner in which it is constructed in our original method, and even if one could there
is no base change theorem to tell you what this gives. 

In this paper we consider the case in which $\Xp_0$ is (almost) any projective hypersurface, perhaps singular
or reducible, or even not reduced as a scheme. In fact we work in slightly greater generality than 
above, but subject to one lifting condition.  We consider a smooth family
 $\Xp \rightarrow \Sp \subseteq \Aff^1_{\sfq}$
of projective hypersurfaces of degree $d$ and dimension $n$ which lifts to a smooth family 
$\XOL \rightarrow \SOL$ of projective hypersurfaces of degree $d$ over the ring of integers $\OL$ of
an algebraic number field $\LL \subset \K$ such that
all fibres in the punctured $p$-adic open unit disk
around the origin are ``smooth and in general position'' (Definition \ref{Def-GenPos}). It is convenient
to further assume that $1 \in \Sp$. We say then that $\Xp \rightarrow \Sp$ is in general position and has a good algebraic
lift around $t = 0$ (Definition \ref{Def-GALhyper}).
Defining as before $\HH := H^n_{rig}(\Xp/\Sp)$, we now introduce a new object 
 $(\HH_0,\N_0,\Fr_0,e)$ which we call the {\it limiting (mixed) Frobenius structure} at $t = 0$.
It is a finite
dimensional vector space $\HH_0$ over the $p$-adic field $\K$ with an invertible $\sigma$-linear operator $\Fr_0: \HH_0 \rightarrow \HH_0$ and nilpotent
operator $\N_0: \HH_0 \rightarrow \HH_0$ with $\N_0^{n+1} = 0$ such that
 \[ \N_0 \Fr_0 = p \Fr_0 \N_0;\]
 and in addition a positive integer $e$. We
call $\Fr_0$ the Frobenius operator and $\N_0$ the monodromy operator. The integer $e$ specifies a
cover $\Sp_e \rightarrow \Sp\,(t \mapsto t^e)$ of the base curve which occurs during the construction.
The monodromy operator $\N_0$ is defined over the algebraic number field $\LL \subset \K$ and so we may represent
it explicitly via a nilpotent matrix of algebraic numbers with respect to some basis of $\HH_0$. 
With respect to the same basis, the matrix for the Frobenius
$\Fr_0$ has entries in the uncountable field $\K$ and one cannot write it down; 
however, for each positive integer $N$ one may approximate it by a matrix $\tilde{\Fr}_0$ with entries in
the algebraic number field $\Kalg$ such that $\ord_p (\Fr_0 - \tilde{\Fr}_0) \geq N$. We call such a $\tilde{\Fr}_0$ a $p^N$-approximation
to the matrix for $\Fr_0$. Our main theoretical result is as follows (Theorem \ref{Thm-LFShypersurfaces}).

\begin{theorem}
The limiting Frobenius structure $(\HH_0,\N_0,\Fr_0,e)$ at $t = 0$
for a degeneration of smooth projective hypersurfaces $\Xp \rightarrow \Sp$ which has a good algebraic lift
around $t = 0$ and is in general position is effectively computable, in the following sense.
Given such a degeneration and any positive integer $N$, one may compute $e$, a basis for $\HH_0$, and with respect to this basis the matrix for $\N_0$ and a $p^N$-approximation to a matrix for $\Fr_0$.
\end{theorem}

Our algorithm for computing the limiting Frobenius structure is deterministic and explicit. It is practical in two senses. 
First, one may show that it is
polynomial-time in $p,\log_p(q),(d+1)^n$ and $N$ under favourable assumptions (Section \ref{Sec-Complexity}).
Second, the author has implemented the algorithm in the programming language
{\sc Magma}, and executed it for degenerations of curves, surfaces and threefolds of various different degrees
(Section \ref{Sec-Exs}).

Given that one can define and compute the limiting Frobenius structure, it is natural to next ask for its geometric significance.
When $0 \in \Sp$ then the fibre $\Xp_0$ in $\Xp \rightarrow \Sp$ is smooth. In this case
$e = 1$, $\N_0 = 0$ and $\Fr_0$ is the usual Frobenius map on $\HH_0 = H^n_{rig}(\Xp_0)$.
Let us suppose then that $0 \not \in \Sp$ and so $\Xp_0$ is not defined. One may
canonically extend $\Xp \rightarrow \Sp$ to a flat 
family $\Yp \rightarrow \Tp$ of projective hypersurfaces over $\Tp:= \Sp \cup \{0\}$.
Let us denote the fibre over the origin in $\Yp \rightarrow \Tp$ as $\Xp_0$ and call it the degenerate fibre.
One sees from some examples (Examples \ref{Ex-DoubleConic}, \ref{Ex-3cusps}, \ref{Ex-Quintic} and \ref{Ex-QuarticCone}) 
that there may be
no relationship between the limiting Frobenius structure and the zeta function of $\Xp_0$. 
Before proceeding further we need some more definitions.

Let us denote by $\Sp_e \rightarrow \Sp$ and $\Tp_e \rightarrow \Tp$ the covers defined by the map $t \mapsto t^e$, and
write $\Xp_e := \Xp \times_{\Sp} \Sp_e$.
(We suspect that it is a consequence of our lifting hypothesis that the integer $e$ is always coprime to $p$ (Note \ref{Note-pande}), and so the cover $\Sp_e \rightarrow \Sp$ is unramified.)
Note that $\Xp_e \rightarrow \Sp_e$ is still a smooth family of projective hypersurfaces.
We shall say that $\Xp_e \rightarrow \Sp_e$ extends to a semistable degeneration over the origin if the following
is true. There exists a smooth $\fq$-variety $\Yp_e$ and flat morphism $\Yp_e \rightarrow \Tp_e$ which restricts
over $\Sp_e \subset \Tp_e$ to $\Xp_e \rightarrow \Sp_e$ with the fibre $\Yp_{e,0}$ over $0 \in \Tp$ a reduced
union of smooth hypersurfaces on $\Yp_e$ crossing transversally. When $\Xp_e \rightarrow \Sp_e$ extends to a semistable
degeneration $\Yp_e \rightarrow \Tp_e$ over the origin, we shall sometimes denote the degenerate fibre $\Yp_{e,0}$ by
$\Xp_{e,0}$, bearing in mind always that it depends upon the choice of extension.
We make the following conjecture, 
motivated in part by an analogy with degenerations of complex algebraic varieties (Theorem \ref{Thm-CS}).
(See Conjecture \ref{Conj-CS} for a fuller statement of our conjecture.)

\begin{conjecture}\label{Conj-CSIntro}
The morphism $\Xp_e \rightarrow \Sp_e$ extends to
a semistable degeneration over the origin with degenerate fibre $\Xp_{e,0}$. Moreover, for any such
extension there exists a ``Clemens-Schmidt'' exact sequence in rigid cohomology containing the following
\[ \cdots \rightarrow H^n_{rig}(\Xp_{e,0}) \rightarrow \HH_0 \stackrel{\N_0}{\rightarrow} \HH_0(-1) \rightarrow
\cdots .\]
\end{conjecture}

The exactness of the above three term sequence in the middle (``local invariant cycle theorem'')
implies that $\det(1 - T\Fr_0^{\log_p(q)}|\Ker(\N_0))$ occurs as a factor in the ``middle Weil polynomial'' of the zeta function of 
the semistable variety $\Xp_{e,0}$. The main experimental work in this paper is the testing of
our conjecture for some degenerations of low dimensional hypersurfaces (Section 
\ref{Sec-Exs}).

Having shown that the limiting Frobenius structure for a degeneration of hypersurfaces under a
lifting assumption is effectively computable, and provided
some experimental evidence supporting a conjectural geometric interpretation of this structure, next one
might ask if all this is of any use. Let us call the algorithm in this paper the degeneration method.
The author suggests at least two practical applications of this method. 

When the morphism $\Yp \rightarrow \Tp$ canonically extending $\Xp \rightarrow \Sp$
over the origin is semistable, under Conjecture \ref{Conj-CS} the degeneration method computes cohomological information about $\Xp_0$ itself. In particular
the algorithm can (conjecturally) be used to compute the zeta function of 
a nodal plane projective curve, by embedding it as the degenerate fibre in a suitable family.
It is a classical fact that any
irreducible projective curve over an algebraically closed field is birational to a plane projective curve with only nodes 
as singularities \cite[Theorem 1.61]{JK}, and this transformation is effective. Thus given any 
geometrically irreducible projective curve over $\fq$ one can explicitly find a semistable plane model over an extension of $\fq$, and use the degeneration method to
compute the zeta function of this plane model; see Example \ref{Ex-SexticDamiano}.

For a second application of the degeneration method,  recall that irreducible plane curves with a single node 
arise naturally when one is interested in calculating the
L-series of a smooth plane curve defined over the integers. The cohomology of these curves gives
the factors at the bad primes. The degeneration method will compute these factors, under Conjecture
\ref{Conj-CS}; see Note \ref{Note-BadPrimes} for an example.
 
We make two concluding remarks. 
First, the existence of a suitably-defined limiting Frobenius structure at $t = 0$ for
any smooth projective family $\Xp \rightarrow \Sp \subset \Aff^1_{\sfq} \backslash \{0\}$ of varieties
actually follows from the semistable reduction theorem of Kedlaya for overconvergent $F$-isocrystals on
an open curve \cite{KKss}; indeed, directly from its local incarnation the $p$-adic local monodromy theorem.
We have just given an independent and effective proof of its
existence for the case of hypersurfaces subject to a lifting assumption, and given
the object a name; see Note \ref{Note-Names} for a discussion of nomenclature. 
(In fact, we show its existence for
any smooth proper family $\Xp \rightarrow \Sp$ subject to a lifting assumption (Definition
\ref{Def-GL} and Theorem \ref{Thm-LFSexist}). 
But without the lifting assumption the situation is much more difficult, and one needs to allow
different covers of $\Sp$ than just $t \mapsto t^e$.)
Kedlaya's semistable reduction theorem motivated the
author to try to adapt the deformation method to handle degenerations.

Second, there is a close analogy with the theory of degenerations of smooth complex algebraic varieties where one
has a limiting mixed Hodge structure \cite[Chapter 11]{PS}. 
However, our situation is even richer. Section \ref{Sec-LFHS} sketches
the construction of a {\it limiting (mixed) Frobenius-Hodge structure} at $t = 0$ in dimension $m$ attached to
a smooth and proper morphism $\XOL \rightarrow \SOL \subset \Aff^1_{\SOL}$ and
a suitably large rational prime $p$ inert in the ring of integers $\OL$ of a number field $\LL$.
This object is new to the author. It is a quadruple $(\HH_0^\bullet,\N_0,\Fr,e)_\W$ consisting of a
filtered $\W$-module $\HH_0^\bullet$ with operators $\N_0$ and $\Fr_0$ such that $\N_0 \Fr_0 = 
p \Fr_0 \N_0$, $\N_0^{m+1} = 0$ and $\Fr_0 \HH_0^i \subseteq p^i \HH_0$, along with a positive integer $e$. 
(Here $\W$ is the ring of integers of the $p$-adic field $\K$.)
When $e = 1$, the pair ``$(\HH_0^\bullet,\N_0)_{\W} \otimes_\W \C$'' is the usual limiting mixed Hodge structure at $t = 0$ of the degeneration of complex
algebraic varieties, and a similar triple to $(\HH_0,\N_0,\Fr_0)_\W$ occurs in crystalline cohomology (Hyodo-Kato cohomology in mixed characteristic).
This structure is computable in an appropriate sense for a degeneration of hypersurfaces and suitably large prime $p$.
One can then extend Conjecture \ref{Conj-CS} by insisting everything
is defined over $\W$ and there is an exact sequence of ``Frobenius-Hodge structures''.  
The existence of this $W$-lattice structure and Hodge filtration on the limiting
Frobenius structures informed the author's choice of
precisions in the computations in Section \ref{Sec-Exs}. 

The paper is organised in the following manner. Section \ref{Sec-Char0} considers the algebraic
de Rham cohomology of degenerations of varieties over a computable field of characteristic
zero. The first main result is some kind of partly ``effective'' semistable reduction theorem in algebraic
de Rham cohomology (Theorem \ref{Thm-SSR}). We then focus on degenerations of projective
hypersurfaces and give an entirely effective proof of semistable reduction for algebraic de Rham cohomology in this 
situation (Theorem \ref{Thm-SSRhyper}). Section \ref{Sec-Charp} considers degenerations in positive
characteristic and introduces limiting Frobenius structures. After showing the notion is well-defined
(Section \ref{Sec-LFSdefuni}) we use Theorem \ref{Thm-SSR} to show that they exist under a lifting assumption (Theorem \ref{Thm-LFSexist} and Definition \ref{Def-GL}). 
We then apply the theory of Section \ref{Sec-Hypersurfaces} to show
that limiting Frobenius structures are effectively computable in an appropriate sense for
degenerations of hypersurfaces subject to a lifting assumption (Definition \ref{Def-GALhyper} and 
Theorem \ref{Thm-LFShypersurfaces}). Next follows a brief sketch of how one endows
our limiting Frobenius structures with an integral structure and Hodge filtration (Section \ref{Sec-LFHS}).
Section \ref{Sec-CS} presents our conjectural ``Clemens-Schmidt'' exact sequence in rigid cohomology
(Conjecture \ref{Conj-CS}). Finally, in Section \ref{Sec-Exs} we present an assortment of
examples computed by the author which give some support to Conjecture \ref{Conj-CS} and also
illustrate our constructions. 
Although the remaining sections are arranged in their logical order, we
recommend that the reader begins by looking through Section \ref{Sec-Exs} and then gently reverses through the paper
before starting in earnest.

 \section{Degenerations in characteristic zero}\label{Sec-Char0}

Let $\F$ be a field of characteristic zero. We wish to be able to compute with elements of $\F$, so
let us assume that $\F$ is countable and one can effectively perform arithmetic operations in $\F$.

\subsection{Relative algebraic de Rham cohomology}

We refer to \cite{KatzOda} for the definition of relative algebraic de Rham cohomology and the Gauss-Manin
connection.
 
Let $\SF \subseteq \Aff_\F^1$ be an open curve, and
$\XF \rightarrow \SF$ be a smooth morphism of algebraic varieties over $\F$ of relative dimension $n$. 
Fix an integer $m$ with $0 \leq m \leq 2n$ and denote by $H^m_{dR} (\XF/\SF)$ the $m$th
relative algebraic de Rham cohomology group of $\XF \rightarrow \SF$. 

Write $\SF = \Spec(\AF)$ where $\AF = \F[t,1/\Delta(t)]$ for some polynomial $\Delta(t) \in \F[t]$. Then
$H^m_{dR}(\XF/\SF)$ is a locally free $\AF$-module of finite rank, which we may assume to be free since
$\AF$ is a principal ideal domain (see for example \cite[Chapter III.7]{SL}). Let
\[ \nabla_{\frac{d}{dt}}: H^m_{dR}(\XF/\SF) \stackrel{\nabla_{GM}}{\rightarrow} H^m_{dR}(\XF/\SF) \otimes \Omega^1_\AF  
\stackrel{id \otimes \frac{d}{dt}}{\rightarrow} H^m_{dR}(\XC/\SC)\]
denote the Gauss-Manin connection $\nabla_{GM}$ contracted with $\frac{d}{dt} \in \HHom(\Omega^1_\AF,\AF)$.
To ease notation we shall write $\nabla := \nabla_{\frac{d}{dt}}$ and call this the Gauss-Manin connection.
The connection $\nabla$ is Leibniz linear; that is,
\[ \nabla(r + s) = \nabla(r) + \nabla (s),\]
\[ \nabla(a r) = \frac{da}{dt} r + a \nabla (r)\]
for $a \in \AF$ and $r,s \in H^m_{dR}(\XF/\SF)$. 

\subsection{Regularity and semistable reduction}

Let $r$ be the rank of $H^m_{dR}(\XF/\SF)$ and $\Basis$ be a basis. 
Denote by $N(t)$ the matrix for the Gauss-Manin connection acting on this basis. The matrix
$N(t)$ has entries in $\AF$. If one changes basis
by a matrix $H \in \GL_r(\AF)$, then since $\nabla$ is Leibniz linear the new matrix for $\nabla$ w.r.t.
the basis $\Basis_{[H]} := \{Hb\}_{b \in \Basis}$
is
\[ N_{[H]} := H N H^{-1} + \frac{dH}{dt} H^{-1}.\]
Suppose that $\Delta(0) = 0$, so that $N(t)$ may have a pole at $t = 0$. Then the ``regularity
theorem'' guarantees that there exists a matrix $H \in \GL_r(\F((t)))$ such that
$N_{[H]}$ has only a simple pole at $t = 0$, see for example \cite{NMK}. 
We use this theorem to prove a stronger and
effective result.

\begin{theorem}[Regularity]\label{Thm-ExplicitReg}
There exists an explicit deterministic algorithm with the following input and output.
The input is the matrix $N(t)$ for the Gauss-Manin connection $\nabla$ on
$H^m_{dR}(\XF/\SF)$ with respect to some basis $\Basis$, where $\XF \rightarrow \SF$ is a smooth
morphism of varieties over a computable field $\F$. The output is
a matrix $H \in \GL_r(\F[t,t^{-1}])$ such that the matrix $N_{[H]}$ for $\nabla$ w.r.t.
the basis $\Basis_{[H]}$ of $H^m_{dR}(\XF/\SF)$ has only a simple pole at $t = 0$.
\end{theorem}

\begin{proof}
If $\Delta(0) \ne 0$ then we are done since the matrix $N(t)$ will not have a pole at $t = 0$.
So assume $\Delta(0) = 0$.
For $\omega \in H^m_{dR}(\XF/\SF)$, let $\Basis_{\omega} := \{(t \nabla)^i(\omega)\}_{0 \leq i < r }$.
By definition, $\omega$ is a cyclic vector for the differential system $(H^m_{dR}(\XF/\SF),(t\nabla))$ over
$\F((t))$
if and only if $\Basis_\omega$ is a linearly independent set over $\F((t))$. For any given $\omega$ expressed
in the basis $\Basis$ one may check whether $\omega$ is cyclic by computing a determinant since
one can calculate the action of $t\nabla$ via $t\frac{d}{dt} + tN(t)$. Now a cyclic
vector exists amongst a certain finite list of $\F[t]$-linear combinations of the basis
$\Basis$ according to \cite[Lemma 2.10]{vPS}. Thus one may effectively find a cyclic vector, $\omega$ say.
The differential system is regular, by the regularity theorem.
Thus by \cite[Corollary 7.1.3]{KKbook},
the matrix for the connection w.r.t. the basis $\Basis_{\omega}$ of
$H^m_{dR}(\XF/\SF) \otimes_{\AF} \F((t))$
has only a simple pole at $t = 0$. We can compute the change of basis matrix $G$ between 
$\Basis$ and $\Basis_{\omega}$, since we can explicitly compute the action of $(t\nabla)^i$ on $\omega$. 
Note that $G \in \GL_r(\F(t))$. Embedding $\F(t) \hookrightarrow \F((t))$ and
using Lemma \ref{Lem-Factor}, we may factor $G = LH$ where $H \in \GL_r(\F[t,t^{-1}])$ and $L \in \GL_r(\F[[t]])$.
Then $N_{[H]}$ has only a simple pole at $t = 0$, since $N_{[H]} = (N_{[G]})_{[L^{-1}]}$ with
$N_{[G]}$ having a simple pole and $L^{-1}$ no pole at $t = 0$. Moreover,
$\Basis_{[H]}$ is a global basis for $H^m_{dR}(\XF/\SF)$ since $H(t) \in \GL_r(\F[t,t^{-1}]) \subseteq \GL_r(\AF)$.
Thus $H$ is the desired change of basis matrix.
\end{proof}
 
 \begin{lemma}\label{Lem-Factor}
 Given any $G \in \GL_r(\F((t)))$ there exists $H \in \GL_r(\F[t,t^{-1}])$ and $L \in \GL_r(\F[[t]])$ such that $G = LH$. Moreover, if $G \in \GL_r(\F(t)) \subset \GL_r(\F((t)))$ then such $H$ and $L \in \GL_r(\F[[t]] \cap
 \F(t))$ are effectively computable.
 \end{lemma}
 
 \begin{proof}
 Let $G \in \GL_r(\F((t)))$. We prove by induction on $r$ that there exists
 $L \in \GL_r(\F[[t]])$ such that $LG$ is
 upper triangular with monomials on the diagonal and elements in $\F[t,t^{-1}]$ above the diagonal.
 The case $r = 1$ is immediate. For $r > 1$, by row operations (equivalently, 
 premultiplying by a suitable matrix in $\GL_r(\F[[t]])$) 
 one may reduce to the case in which
 $G$ is upper triangular with monomials on the diagonal and elements in $\F((t))$ above the diagonal. Next, by
 induction one may assume that all entries in $G$ above the diagonal except those on the first row lie in $\F[t,t^{-1}]$. 
 Finally, by an explicit calculation one shows that there exists a matrix $L \in \GL_r(\F[[t]])$ of the following form such that
 $LG$ is as required: $L$ equals the identity matrix except for elements in positions $(1,i)$ for $2 \leq i \leq r$. 
 Note that the factorisation is unique up to multiplication by an element of $\GL_r(\F[t])$.
 In the case in which $G \in \GL_r(\F(t))$ one checks that this proof describes an effective algorithm and 
 $L \in \GL_r(\F[[t]] \cap \F(t))$.
 \end{proof}

\begin{note}
The use of a cyclic vector and the factorisation in Lemma \ref{Lem-Factor} were suggested to the author
by Kiran Kedlaya.
Kedlaya also suggests replacing the use of a cyclic vector by a basis over $\F[[t]]$ for the
$(t\nabla)$-stable lattice in $\F((t))$ generated by the elements $\{(t\nabla)^i(b)\}$ for $0 \leq i < r$ and
$b \in \Basis$. We further note that one could use Katz's explicit construction of a cyclic vector \cite[Lemma 2.11]{vPS};
see paragraph 2 in the proof of \cite[Theorem 5.4.2]{KKbook} for how to get around the assumption of Katz
that one has an element ``$z$'' in $\F((t))$ with ``$z^\prime = 1$'' where in our case the dash is the operator
$t\frac{d}{dt}$.
\end{note}

The significance of Theorem \ref{Thm-ExplicitReg} is that using a change of basis matrix which
is invertible over $\F[t,t^{-1}]$ does not introduce any new poles in the matrix for the connection. 
Thus when $\F$ is a countable subfield of
the field of $p$-adic numbers
(or the complex numbers), one does not change the radius of convergence on the punctured disk
around zero of the connection matrix. This is essential in applications of the theorem to $p$-adic cohomology
(Section \ref{Sec-Charp}).

We now turn to the problem of semistable reduction.
For $e$ a positive integer, write $\AF_e :=  \F[s,1/\Delta(s^e)]$ and $\SF_e := \Spec(\AF_e)$.
Define $\AF \rightarrow \AF_e$ by $t \mapsto s^e$ to get a morphism $\SF_e \rightarrow \SF$.
One may pull-back the family $\XF \rightarrow \SF$ via $\SF_e \rightarrow \SF$
to get $\XF_e \rightarrow \SF_e$ where $\XF_e := \XF \times_{\SF} \SF_e$. We have
\[ H^m_{dR}(\XF_e/\SF_e) \cong H^m_{dR}(\XF/\SF) \otimes_\AF \AF_e\]
since $\AF_e$ is a flat $\AF$-module.
For a basis $\Basis$ of $H^m_{dR}(\XF/\SF)$ let $\Basis_e := \{b \otimes 1_{\AF_e}\}$ be
the associated basis of $H^m_{dR}(\XF_e/\SF_e)$. As before let
$N(t)$ be the matrix for the Gauss-Manin connection w.r.t. the basis $\Basis$.
Then the matrix
for the Gauss-Manin connection 
\[ \nabla_e: H^m_{dR}(\XF_e/\SF_e) \stackrel{\nabla_{GM}}{\rightarrow} H^m_{dR}(\XF_e/\SF_e) \otimes \Omega^1_{\AF_e}  
\stackrel{id \otimes \frac{d}{ds}}{\rightarrow} H^m_{dR}(\XF_e/\SF_e)\]
w.r.t. the basis $\Basis_e$ is
\[ (es^{e-1}) N(s^e),\]
since $dt = e s^{e-1} ds$.

The problem of local semistable reduction in algebraic de Rham cohomology is to find a positive integer
$e$ and basis for $H^m_{dR}(\XF_e/\SF_e)$ such that the matrix for the connection $\nabla$ has only
a simple pole at $s = 0$ with nilpotent residue matrix.

\begin{theorem}[Semistable reduction]\label{Thm-SSR}
There exists an explicit deterministic algorithm with the following input and output.
The input is the matrix $N(t)$ for the Gauss-Manin connection on
$H^m_{dR}(\XF/\SF)$ w.r.t. some basis $\Basis$, where $\XF \rightarrow \SF$ is a smooth
morphism of varieties over a computable field $\F$. The output is an integer $e$ and 
matrix $H(s) \in \GL_r(\F[s,s^{-1}])$ such that the matrix
\[ (es^{e-1}N(s^e))_{[H(s)]} := Hes^{e-1}N (s^e) H^{-1} + \frac{dH}{ds} H^{-1} \]
for the Gauss-Manin connection $\nabla_e$ on the basis $(\Basis_e)_{[H(s)]}$ of
$H^m_{dR}(\XF_e/\SF_e)$ has only a simple pole at $s = 0$ with nilpotent residue matrix.
\end{theorem}

\begin{proof}
By Theorem \ref{Thm-ExplicitReg} one reduces to the case in which $N(t)$ has only a simple pole
at $t = 0$. Let $\{\lambda_i\}_{1 \leq i \leq r}$ denote the eigenvalues of the residue matrix. By the local
monodromy theorem we have $\lambda_i \in \Q$, see \cite{NMK}. Thus we may find a smallest positive integer $e$ such that
$e \lambda_i \in \Z$ for all $1 \leq i \leq r$. Then $e s^{e-1}N(s^e)$ has only a simple pole at
$s = 0$ and a residue matrix with integer eigenvalues. Shearing transformations now allow one to find
a matrix $H(s) \in \GL_r(\F[s,s^{-1}])$ such that $(es^{e-1}N(s^e))_{[H]}$ has only a simple pole with
nilpotent residue matrix \cite[Lemma III.8.2]{DGS}. (Note that since in our application the eigenvalues of the
residue matrix lie in $\Z$, the initial reduction to upper triangular form of the residue matrix in the proof of
\cite[Lemma III.8.2]{DGS} can be performed over $\F$.)
\end{proof}

\begin{note}
Since the change of basis matrix $H$ lies in $\GL_r(\F[s,s^{-1}])$ one can perform global semistable reduction over
the whole of the affine basis $\SF$. That is, after pulling back by a chain of unramified covers 
$\SF_e \rightarrow \SF$,  effectively find a matrix for the connection which has at worst simple poles at all 
finite points with nilpotent residue matrices. Moreover, this can be done without introducing any new finite poles to the
differential system. Alternatively, without a base change one can effectively find a matrix for the connection with only simple
poles at all finite points --- this is useful when computing in the cokernel 
\[ H^1_{dR}(\SF,H^m_{dR}(\XF/\SF)) := \Coker(\nabla_{GM})\] 
of the connection $\nabla_{GM}$ \cite[Section 4]{L1}.
\end{note}

\subsection{Hypersurfaces}\label{Sec-Hypersurfaces}

The first significant result of this section is Theorem \ref{Thm-GMDwork}, on the construction of the Picard-Fuchs
differential system for families of projective hypersurfaces. Our construction
is based upon the method of Griffiths and Dwork. However, it differs
from the usual presentation in the literature (for example \cite[Section 5.3]{CK}) in one key way. We
avoid Gr\"{o}bner basis computations, and instead work with linear algebra and Macaulay-style resultant matrices. This has a number of benefits. First, it allows us to give an explicit a priori characterisation of the
locus of poles of the Picard-Fuchs system w.r.t. a particular basis, and thus prove a
theorem on explicit semistable reduction for hypersurfaces (Theorem \ref{Thm-SSRhyper}).
Second, it reveals additional structure in the Picard-Fuchs system (Corollary \ref{Cor-PFstructure}).
Third, in the author's own experience, it is faster in
practical implementations (Note \ref{Note-Pancratz}). Combining Theorems \ref{Thm-SSR} and \ref{Thm-GMDwork}
yields Theorem \ref{Thm-SSRhyper}, the main result of Section \ref{Sec-Char0}.

\subsubsection{Cohomology of the generic hypersurface}

In this section we fix $n \geq 2$ and $d \geq 0$ and let $x_0,x_1,\ldots,x_{n+1}$ be variables.
For $w = (w_0,w_1,\dots,w_{n+1}) \in \Z_{\geq 0}^{n+2}$ define $x^w := x_0^{w_0} x_1^{w_1} \cdots x_{n+1}^{w_{n+1}}$ and
$|w| := w_0 + w_1 + \cdots + w_{n+1}$. For each $w \in \Z_{\geq 0}^{n+2}$  with $|w| = d$ we introduce a variable
$a_w$ and define $\PR := \sum_{|w| = d} a_w x^w$ to be the ``generic polynomial'' homogeneous of degree $d$ with
coefficients in the multivariate polynomial ring $\UnivR := \Z[a_w\,|\, |w| = d]$. Write 
\[ \Omega := \sum_{i = 0}^{n+1} (-1)^i x_i dx_0 \wedge \cdots \wedge \widehat{dx_i} \wedge \cdots \wedge dx_{n+1}\]
where the hat denotes omission, and define
\[ \Basis_\UR := \left\{ \frac{x^w \Omega}{\PR^k}\,|\,0 \leq w_i < d-1 \mbox{ and } \exists k \in \Z \mbox{ s.t. } |w| + n + 2 = kd \right\}.\] 
Elements in $\Basis_\UR$ are differential forms on a certain affine scheme $\UR$ which is defined
in the statement of the following theorem.

\begin{theorem}\label{Thm-GP}
There exists an explicitly computable non-zero polynomial $\Delta_{n,d} \in \UnivR$ such that
the following is true. Define
$\UnivR^{(1)} := \UnivR[1/\Delta_{n,d}]$ and $\SR := \Spec(\UnivR^{(1)})$. Let $\XR \subset \pr^{n+1}_{\SR}$ be the
projective hypersurface defined by the polynomial $P$, and let $\UR$ be the affine complement of $\XR$ in $\pr^{n+1}_{\SR}$. Then the set of
cohomology classes $\{[b]\,|\, b \in \Basis_\UR\} \subset H^{n+1}_{dR}(\UR/\SR)$ is an $\UnivR^{(1)}$-basis for
$H^{n+1}_{dR}(\UR/\SR)$.
\end{theorem}

\begin{proof}
This is an explicit computation carried out in the Section \ref{Sec-Red}.
\end{proof}

\begin{definition}\label{Def-GenPos}
We shall say that a hypersurface of degree $d$ and dimension $n$ over a field of characteristic zero
is smooth and in general position if its coefficients are not zeros of $\Delta_{n,d}$.
\end{definition}

It will follow immediately from our definition of $\Delta_{n,d}$ in Section \ref{Sec-Red} that a hypersurface which is
``smooth and in general position'' is indeed smooth.

With the notation of the theorem above, there is an algebraically defined map called the residue map
\[ H^{n+1}_{dR}(\UR/\SR) \stackrel{{\rm Res}}{\rightarrow} H^n_{dR}(\XR/\SR). \]
It is injective and the image is by (our) definition the primitive middle-dimensional cohomology
\[ \HH_\SR := {\rm Prim}\, H^n_{dR}(\XR/\SR).\]
Note that $\HH_\SR = H^n_{dR}(\XR/\SR)$ when $n$ is odd, and $H^n_{dR}(\XR/\SR)/\HH_\SR$
is free of rank $1$ when $n$ is even. Define
\[ \Basis_\XR := \{ {\rm Res}([b]): b \in \Basis_\UR\}.\]

\begin{corollary}\label{Cor-DworkBasis}
With the notation of Theorem \ref{Thm-GP},
 the set of residues $\Basis_\XR$ is an $\UnivR^{(1)}$-basis for the primitive middle-dimensional algebraic
de Rham cohomology $\HH_\SR$ of $\XR \rightarrow \SR$.
\end{corollary}

We shall call $\Basis_\XR$ the Dwork basis of $\HH_\SR$.

\subsubsection{From the universal to the generic hypersurface}\label{Sec-Red}

This section contains our explicit proof of Theorem \ref{Thm-GP}. We examine the primitive middle-dimensional
algebraic de Rham cohomology
of a hypersurface, starting with a universal one and then specialising to the generic smooth hypersurface and
then the hypersurface which is ``smooth and in general position''. Here ``in general position'' is a utilitarian
definition which allows us to write down a certain basis for cohomology.

Let us introduce some new notation for use only in this section. 
For the ring $\UnivR = \Z[a_w\,|\, |w| = d]$ let $\pr^{n+1}_\UnivR$ denote projective
space of dimension $n+1$ over $\UnivR$. Let $\XR/\UnivR$ 
be the ``universal'' projective hypersurface of degree $d$ in $\pr^{n+1}_\UnivR$ defined
by the polynomial $P = \sum_{|w| = d}  a_w x^w$. Denote by $\UR/\UnivR$ its affine complement. Then
\[ H^\bullet_{dR}(\UR/\UnivR) := \Hyper (\Omega^\bullet_{\UR/\UnivR}) = 
H(\Gamma(\UR/\UnivR,\Omega^\bullet_{\UR/\UnivR})).\]
Here $\Hyper(\cdot)$ denotes hypercohomology of a complex of sheaves, $H(\cdot)$ 
homology of a complex of modules, and the second equality follows since
$\UR/\UnivR$ is affine and so the \v{C}ech-de Rham spectral sequence for computing hypercohomology
degenerates at the first term. In particular since $\Omega^{n+2}_{\UR/\UnivR} = 0$, 
\[ H^{n+1}_{dR}(\UR/\UnivR) = \frac{\Gamma(\UR/\UnivR,\Omega^{n+1}_{\UR/\UnivR})}{d\Gamma(\UR/\UnivR,\Omega^{n}_{\UR/\UnivR})},\]
where $d$ is derivation of $n$-forms. That is, 
$H^{n+1}_{dR}(\UR/\UnivR)$ is global (closed) $(n+1)$-forms modulo exact ones.

Here is an explicit description of $H^{n+1}_{dR}(\UR/\UnivR)$ --- this is due to Griffiths but a convenient
reference is \cite[Section 3]{AKR}.
For $k \geq 1$ define
\[ Z_k := \left\{ \frac{A \Omega}{\PR^k}\,|\,A \in \UnivR[x_0,\cdots,x_{n+1}] \mbox{ homogeneous degree } kd - n - 2\right\}. \]
Then $Z_k$ is the $R$-module of $(n+1)$-forms on $\pr^{n+1}_{\UnivR}$ defined on $\UR/\UnivR$ with
poles of order at most $k$ along $\XR/\UnivR$; that is,
\[Z_k = \Gamma(\pr^{n+1}_{\UnivR},\Omega^{n+1}_{\spr^{n+1}_{\UnivR}}(k (\XR/\UnivR))).\]
Note $Z_k \subset Z_{k+1}$ for $k \geq 1$.
Defining $Z := \cup_{k = 1}^\infty Z_k$ we have $Z = \Gamma(\UR/\UnivR,\Omega^{n+1}_{\UR/\UnivR})$ with
the natural filtration by pole order.

Let $b_1 = \{0\}$ and for $k \geq 2$ define $b_k$ to be the $\UnivR$-module generated
by the set
\[ \left\{ \frac{(\delta_i A) \Omega}{\PR^{k-1}} - (k-1) \frac{A (\delta_i P) \Omega}{\PR^{k}} \,|\,
A \in \UnivR[x_0,\cdots,x_{n+1}] \mbox{ homo. deg. } (k-1)d - n - 1\right\}\]
where $\delta_i := \frac{\partial }{\partial x_i}$.
Let $B_k$ be the $\UnivR$-module generated by $\cup_{j = 1}^k b_j$. Then
\[ B_k = d\Gamma(\pr^{n}_{\UnivR},\Omega^{n}_{\spr^{n+1}_{\UnivR}}((k-1) (\XR/\UnivR)))\]
where $d$ is derivation of $n$-forms.
Defining $B := \cup_{k = 1}^\infty B_k$ we have
$B = d\Gamma(\UR/\UnivR,\Omega^{n}_{\UR/\UnivR})$ filtered by pole order. Note that 
$B_k \subseteq B \cap Z_k$ but there is not an equality: it fails for singular
hypersurfaces, see \cite[Theorem A, (A2)]{Dimca} and the comments which follow. 
Certainly though $H^{n+1}_{dR}(\UR/\UnivR) = Z/B$.

Since we are working with the universal hypersurface, the quotient $Z/B$ though presumably
finitely generated is complicated. However, by inverting a sequence of polynomials in
$\UnivR$ one can give an explicit procedure for computing in this quotient based upon reduction of 
pole order. This is the Griffiths-Dwork method. It has been given a conceptual explanation by Dimca.
Filtration by pole order of $\Omega^\bullet_{\UR/\UnivR}$ allows one to define a spectral sequence with $E_\infty = 
H^{n+1}_{dR}(\UR/\UnivR)$. This spectral sequence has infinitely many non-zero $E_1^{s,t}$ terms 
for the universal hypersurface, but it degenerate at the $E_1$ term ($E_1 = E_\infty$) for smooth hypersurfaces (when
``$B_k = B \cap Z_k$'') \cite[Page 764]{Dimca}.
We now give the explicit description.

Let $\Delta^{(n+1)}_{n,d}$
be the ``Macaulay resultant'' of the partial derivatives of $P$ defined in \cite[Page 7, Chapter I.6]{Mac}.
Precisely, let it be Macaulay's 
determinant ``$D$'', taking ``$n$'' to be $n+2$ and ``$F_i$'' to be $\delta_i  \PR$.
Thus $\Delta_{n,d}^{(n+1)}$ is non-zero but vanishes on the coefficients of homogeneous polynomials
of degree $d$ which define singular hypersurfaces of dimension $n$ (and also possibly on the coefficients
of some other polynomials). 
Define $\UnivR^{(n+1)} := \UnivR[1/\Delta^{(n+1)}_{n,d}]$ and for any $\UnivR$-module
$\bullet$ let $\bullet^{(n+1)} := \bullet \otimes_{\UnivR} \UnivR^{(n+1)}$. 
Then we have
\[ Z_k^{(n+1)} = B_k^{(n+1)} + Z^{(n+1)}_{n+1} \mbox{ for }k \geq n+1.\]
This follows since after inverting $\Delta^{(n+1)}_{n,d}$ one may use linear algebra (via
the matrix whose determinant is the Macaulay resultant) to reduce any
$(n+1)$-form in $Z_k^{(n+1)}$ modulo exact forms (in $B^{(n+1)}_k$) to have pole order at most $n+1$. It
follows 
\[ Z^{(n+1)}/B^{(n+1)} \cong Z_{n+1}^{(n+1)}/(B^{(n+1)} \cap Z_{n+1}^{(n+1)})\]
and so $Z^{(n+1)}/B^{(n+1)}$ is finitely generated. Note that since certainly $B_{n+1}^{(n+1)} \subseteq
B^{(n+1)} \cap Z_{n+1}^{(n+1)}$ we have $Z^{(n+1)}/B^{(n+1)}$ is a quotient of 
$Z^{(n+1)}_{n+1}/B^{(n+1)}_{n+1}$; in fact, $B_{n+1}^{(n+1)} =
B^{(n+1)} \cap Z_{n+1}^{(n+1)}$ by \cite[Theorem A, (A2)]{Dimca} and the two are equal.

The quotient $Z^{(n+1)}/B^{(n+1)}$ is locally free, but it is not free in general without further localisation which
will depend upon a choice of basis. Now $\Basis_{\UR}$ is a basis generically since it is a basis
for the de Rham cohomology of the complement of the diagonal hypersurface 
$x_0^d + \cdots + x_{n+1}^d = 0$. (Over $\C$ the action of the $(n+2)$-fold product of the group of $d$th roots of unity
decomposes into characters via this basis.)
 Write $\Basis_{\UR,k} := \Basis_{\UR} \cap Z_k$, an intersection
as sets. One can now construct an explicit Macaulay-style matrix depending on $\Basis_{\UR,n+1}$ whose determinant $\Delta^{(n)}_{n,d}$ is such that defining $\UnivR^{(n)} := \UnivR^{(n+1)}[1/\Delta^{(n)}_{n,d}]$
one has
\[ Z^{(n)}/B^{(n)} \cong Z^{(n)}_{n+1}/B^{(n)}_{n+1} \cong \langle \Basis_{\UR,n+1} \rangle_{\UnivR^{(n)}} \oplus
\left( Z^{(n)}_{n}/B^{(n)}_{n} \right).\]
Here we are following the same convention for base extension of an $\UnivR$-module $\bullet$, that is,
$\bullet^{(n)} :=  \bullet \otimes_\UnivR \UnivR^{(n)}$. In words, every $(n+1)$-form with a pole
of order at most $n+1$ can be written modulo an exact form of pole order at most $n+1$ 
as an $\UnivR^{(n)}$-linear combination of the basis forms in $\Basis_{\UR,n+1}$ plus a form
of pole order at most $n$.

This continues inductively, defining $\Delta^{(k)}_{n,d}$ and $\UnivR^{(k)} := \UnivR^{(k+1)} [1/\Delta^{(k)}_{n,d}]$ for $k = n-1,\ldots,1$ at each step so that
\begin{equation}\label{Eqn-ZB}
Z^{(k)}/B^{(k)} \cong \langle \cup_{\ell = k+1}^{n+1}\Basis_{\UR,\ell} \rangle_{\UnivR^{(k)}} \oplus
\left( Z^{(k)}_{k}/B^{(k)}_{k} \right)
\end{equation}
and this isomorphism is effectively computable via linear algebra. Finally noting
$Z^{(1)}_{1}/B^{(1)}_{1} \cong \langle B_{\UR,1} \rangle_{\UnivR^{(1)}}$ one finds
\[ Z^{(1)}/B^{(1)} \cong \langle \Basis_{\UR} \rangle_{\UnivR^{(1)}}.\]
Finally define $\Delta_{n,d}:= \prod_{k = 1}^{n+1} \Delta_{n,d}^{(k)}$ so that
$\UnivR^{(1)} = \UnivR[1/\Delta_{n,d}]$, and let $\SR := \Spec(\UnivR^{(1)})$. Writing
``$\UR/\SR$'' to denote $\UR \otimes_\UnivR \UnivR^{(1)}/\UnivR^{(1)}$ one has Theorem \ref{Thm-GP}.

\begin{note}\label{Note-Pancratz}
The exact description of the Macaulay-style
matrices whose determinants are the polynomials $\Delta^{(k)}_{n,d}$, and the method
for using these matrices to make the isomorphisms above explicit is a little involved; see
\cite{Lfocm} for the details in a closely analogous situation, or \cite{SP} for full details
of this construction. The method has
been implemented in code by S. Pancratz and seems well-suited to
computing Picard-Fuchs systems for families of hypersurfaces; see \cite{SP} and the proof of
Theorem \ref{Thm-GMDwork}. Our computations of Picard-Fuchs systems for the examples in
Section \ref{Sec-Exs} actually use a slower implementation due to the author, which exploits
the Gr\"{o}bner basis routines for the rational function field $\Q(t)$ in the programming
language {\sc Magma}. 
\end{note}

\subsubsection{Computation of the Picard-Fuchs system}\label{Sec-CPFS}

Let $\F$ be a computable field of characteristic zero.
Let $\PF_t \in \F[t][x_0,x_1,\cdots,x_{n+1}]$ be a homogeneous polynomial of degree $d$ in the
variables $x_0,x_1,\ldots,x_{n+1}$ with coefficients in the polynomial ring $\F[t]$. Let
$\YF_t \subset \pr_{\F(t)}^{n+1}$ be the hypersurface defined by the equation $\PF_t = 0$. 
We shall assume that $\YF_t$ is smooth and in general position, see Definition \ref{Def-GenPos}.
Denote by $\Delta_{n,d}(\PF_t) \in \F[t]$ the non-zero polynomial obtained by specialising $\Delta_{n,d}$
at the coefficients of $\PF_t$.
Let $\Delta(t) \in \F[t]$ be any non-zero polynomial which contains $\Delta_{n,d}(\PF_t)$ as
a factor. Define $\AF := \F[t,1/\Delta(t)]$ and
$\SF := \Spec(\AF)$. Let $\XF \subset \pr^{n+1}_{\SF}$ be the scheme defined by
$\PF_t = 0$ and $\UF$ be its affine complement. 

\begin{note}
One is really just interested in the case $\Delta(t) := \Delta_{n,d}(\PF_t)$ but it is harmless to
excise some additional fibres and allowing this eases the statements of some results.
\end{note}

\begin{corollary}
The set of cohomology classes
$\{[b]\,|\,b \in \Basis_{\UR}\}$ is an $\AF$-basis for $H^{n+1}_{dR}(\UF/\SF)$.
\end{corollary}

Here the $P$ in the definition of $\Basis_{\UR}$ should be specialised to $\PF_t$.

\begin{proof}
By Theorem \ref{Thm-GP}, noting
$H^{n+1}_{dR}(\UR/\UnivR)$ is a quotient $Z/B$ and tensor product of rings
$\bullet \otimes_{\UnivR^{(1)}} \AF$ is right exact.
\end{proof}

Letting $\HH$ denote the primitive part of $H^n_{dR}(\XF/\SF)$ we deduce the following. 

\begin{corollary}
The Dwork basis $\Basis_\XR$ is an $\AF$-basis for $\HH$.
\end{corollary}

We can now present our theorem on Picard-Fuchs systems for projective hypersurfaces.

\begin{theorem}\label{Thm-GMDwork}
The matrix $N(t)$ for the Gauss-Manin connection $\nabla$ restricted to the primitive middle-dimensional
cohomology $\HH$ of $\XF \rightarrow \SF$ with respect to the Dwork basis $\Basis_\XR$ is effectively computable.
\end{theorem}

\begin{proof}
Equivalently, we compute the Gauss-Manin connection on the basis of $H^{n+1}_{dR}(\UF/\SF)$
given by classes of forms in $\Basis_\UR$.
Partition the basis of differential forms $\Basis_\UR$ via pole order $k$ along $\XF$ as before,
\[ \Basis_\UR = \bigsqcup_{k = 1}^{n+1} \Basis_{\UR,k}.\]
For $0 \leq k \leq n$, let $\HH^{k}$ denote the $\AF$-module spanned by the image of $\sqcup_{j = 1}^{k+1} \Basis_\UR^j$ in
the primitive middle-dimensional cohomology $\HH$ of $\XF \rightarrow \SF$. Then 
$\HH^\bullet$ is the Hodge filtration on $\HH$. The Gauss-Manin connection $\nabla = \nabla_{\frac{d}{dt}}$ is
just differentiation of cohomology classes w.r.t. the parameter followed by reduction in cohomology.
Computing $\frac{d}{dt}$ on a basis element is straightforward:
\begin{equation}\label{Eqn-ktokplus1}
 \left[ \frac{x^w \Omega}{P_t^k} \right] \stackrel{\frac{d}{dt}}{\longmapsto} \left[ \frac{-kx^w \frac{dP_t}{dt}\Omega}{P_t^{k+1}} \right],
\end{equation}
where square brackets indicates the cohomology class of a differential form.
Reduction in cohomology is performed via linear algebra as in Section \ref{Sec-Red}, by specialising all of the matrices
$\Delta^{(k)}_{n,d}$ at the coefficients in $\F[t]$ of the polynomial $\PF_t$.
\end{proof}

This explicit description of the computation of the Picard-Fuchs matrix $N(t)$ via linear algebra shows that
it has a particular block form. Specifically, for $1 \leq k \leq n+1$ let $\Delta^{(k)}_{n,d}(\PF_t)$ denote the specialisation of the polynomial $\Delta^{(k)}_{n,d}$ from Section \ref{Sec-Red} 
at the coefficients of $\PF_t$. Define
$\Delta^{(0)}_{n,d}(\PF_t) := 1$.

\begin{corollary}\label{Cor-PFstructure}
Let $N(t)$ be the Picard-Fuchs matrix w.r.t. the Dwork basis, and
$N_{k,\ell}\,(1 \leq k,\ell \leq n+1)$ the block of entries which gives for $b \in \Basis_{\UR,\ell}$ the coefficients 
in $\nabla(b)$ of the elements in $\Basis_{\UR,k}$. Then for $k - 1 \leq \ell$ the matrix 
\[ \left( \prod_{j = k -1}^{\ell} \Delta^{(j)}_{n,d}(\PF_t) \right) N_{k,\ell}\]
has entries in $\F[t]$.  
For $k > \ell + 1$ we have $N_{k,\ell} = 0$ (Griffiths's  transversality). 
\end{corollary}

\begin{proof}
Let $b \in \Basis_{\UR,\ell}$. Then by (\ref{Eqn-ktokplus1}), $\nabla(b)$ can be represented by a differential form
with pole order $\ell + 1$ along $\XF$. Now use the explicit pole reduction method at the end of
Section \ref{Sec-Red}, specialising $\PR$ to $\PF_t$ and tensoring $\UnivR$-modules by $\bullet \otimes_\UnivR \F[t]$,
to write this as an $\F(t)$-linear combination of the basis $\Basis_{\UR}$, dividing by
$\Delta^{(j)}_{n,d}(\PF_t)\,(1 \leq j \leq \ell)$ when reducing pole orders from $j + 1$ to $j$ modulo exact forms. 
Precisely, when one has reduced modulo exact forms to a get pole of order $k$ plus an $\F[t,1/\prod_{j = k}^{\ell} \Delta^{(j)}_{n,d}(\PF_t)]$-linear combination of $\cup_{j = k+1}^{\ell} \Basis_{\UR,j}$, 
one needs to divide by $\Delta^{(k-1)}_{n,d}(\PF_t)$ to
obtain an $\F[t,1/\prod_{j = k-1}^{\ell} \Delta^{(j)}_{n,d}(\PF_t)]$-linear combination of $\cup_{j = k}^{\ell}\Basis_{\UR,j}$ modulo an exact form and a form of pole order $k - 1$, cf. (\ref{Eqn-ZB}).
\end{proof}

In particular, provided $\Delta_{n,d}(\PF_t)$ is squarefree then the Picard-Fuchs matrix has simple poles at the
finite points. This is true generically if $\Delta_{n,d}$ itself is squarefree, which in turn follows if one finds one
specialisation of maximum degree which is squarefree; that is, one suitable family.
The author expects that a Lefschetz pencil would work, but he has not proved this.

\subsubsection{Semistable reduction for hypersurfaces}

We come to the main theorem of Section \ref{Sec-Hypersurfaces}. Recall that for $e$ a positive integer
and $\SF = \Spec(\F[t,1/\Delta(t)])$, we define $\SF_e := \Spec(\F[s,1/\Delta(s^e)])$ with
$\SF_e \rightarrow \SF$ defined by $t \mapsto s^e$ and write $\XF_e := \XF \times_{\SF} \SF_e$.

\begin{theorem}\label{Thm-SSRhyper}
There exists an explicit deterministic algorithm with the following input and output. The input to the algorithm is a 
homogeneous polynomial
$\PF_t  \in \F[t][x_0,x_1,\dots,x_n]$ and a polynomial $\Delta(t) \in \F[t]$
such that writing $\SF := \Spec(\F[t,1/\Delta(t)])$ the fibres in the family $\XF \rightarrow \SF$ defined by
$\PF_t$ are smooth and in general position. The output is 
 a positive integer $e$, a basis
$\Basis_e$ for the primitive middle-dimensional cohomology of $\XF_e \rightarrow \SF_e$ and
a matrix $N(t)$ for the Gauss-Manin connection $\nabla_e$ on this basis which has only a simple
pole at $s = 0$ with nilpotent residue matrix.
\end{theorem}

\begin{proof}
By Theorem \ref{Thm-GMDwork} one can compute the matrix for the Gauss-Manin connection
$\nabla$ on the primitive middle-dimensional cohomology of $\XF \rightarrow \SF$. If
$\Delta(0) \ne 0$ then we are done. Otherwise, Theorem \ref{Thm-SSR} performs the necessary
pull-back $\SF_e \rightarrow \SF$ and change of basis defined over the whole of $\SF_e$. 
\end{proof}

\section{Degenerations in positive characteristic}\label{Sec-Charp}

Let $\fq$ be the finite field with $q$ elements of characteristic $p$, $\K$ be the
unique unramified extension of the field of $p$-adic numbers $\Qp$ of degree $\log_p(q)$, and
$\W$ be the ring of integers of $\K$. Let $\bar{\K}$ denote an algebraic closure of $\K$.

\subsection{Relative rigid cohomology}\label{Sec-RRC}

Let $\XW \rightarrow \SW$ be a 
morphism of $\W$-schemes where $\SW$ is an open subscheme of $\Aff_\W^1$.

\begin{definition}\label{Def-SP}
We shall say that ``Assumption (SP) holds'' if either $\XW \rightarrow \SW$ is smooth and proper,
or it is the restriction of a smooth and proper morphism to the complement of a normal crossings divisor with
smooth components over $\SW$.
\end{definition}

We shall assume throughout Section \ref{Sec-Charp} that Assumption (SP) holds.
Write $\SW := \Spec(\AW)$ where $\AW = \W[t,1/\Delta(t)]$ for some polynomial $\Delta(t) \in \W[t]$.
We assume that $p$ does not divide $\Delta(t)$.
Let
$\XK \rightarrow \SK$ and $\Xp \rightarrow \Sp$ denote maps on generic and special fibres, respectively.
Let $\AK := \AW \otimes_\W \K$ and
define $\AdagK$ to be the weak completion of $\AW$ tensored by $\K$. Precisely, denote the completion of
$\AW$ w.r.t. the Gaussian norm \cite[Page 10]{DGS} by $\hat{\AW}$. Then $\AdagK$ consists of series in $\hat{\AW} \otimes_\W \K$ which converge on
$\pr^1_{\bar{K}}$ with open disks of some unspecified radius less than one around the roots of $\Delta(t)$ and
infinity removed. Thus
$\AdagK$ has a $p$-adic valuation $\ord_p(\cdot)$, and one may extend this to matrix rings over $\AdagK$
in the obvious way. Define
the Frobenius $\sigma: \AdagK \rightarrow \AdagK$ by $\sigma: t \mapsto t^p$ and let
it act on coefficients by the usual Frobenius automorphism of $\K$.

Let $H^m_{rig}(\Xp/\Sp)$ denote the $m$th relative rigid cohomology of $\Xp \rightarrow \Sp$.
(We refer to \cite[Section 2]{Gerk}, \cite[Section 3]{L1} and \cite{LeStum}
for details and references on the construction of relative rigid cohomology.)
This is a locally
free $\AdagK$-module of finite rank equipped with a connection $\nabla_{rig}$ and $\sigma$-linear 
Frobenius $\Fr$ which commute, in the following sense. By Assumption (SP) we have the comparison 
theorem \cite[Theorem 3.1]{L1}
\begin{equation}\label{assumption}
H^m_{rig}(\Xp/\Sp) \cong H^m_{dR}(\XK/\SK) \otimes_{\AK} \AdagK.
\end{equation}
Since $H^m_{dR}(\XK/\SK)$ is a locally free module over a principal ideal domain it is free.  
Thus
$H^m_{rig}(\Xp/\Sp)$ is actually a free $\AdagK$-module.
Let $\Basis$ be a basis for $H^m_{dR}(\XK/\SK)$
and $N(t)$ the matrix for the Gauss-Manin connection $\nabla$ w.r.t. $\Basis$. Write
$\Basis_{rig} := \{b \otimes 1_{\AdagK}\}$ for the associated basis of $H^m_{rig}(\Xp/\Sp)$.
Then by the comparison theorem the matrix for $\nabla_{rig}$ w.r.t. $\Basis_{rig}$ is also $N(t)$.
Let $F(t)$ denote the matrix for Frobenius $\Fr$  w.r.t. $\Basis_{rig}$. Thus
$F(t)$ has entries in $\AdagK$. The matrix $F(t)$ is invertible and
\begin{equation}\label{Eqn-FN}
\frac{dF}{dt} + N(t)F(t) = F(t) \sigma(N(t))p t^{p-1}.
\end{equation}
This equation exactly describes the ``commutativity'' of $\nabla_{rig}$ and $\Fr$.

Denote by $\Robba$ the Robba ring around $t = 0$ \cite[Chapter 15.1]{KKbook}. 
That is, the ring of two-way infinite Laurent series with coefficients
in $\K$ which converge on some open annulus of outer radius one and unspecified inner radius $r < 1$ around the origin.
From the definition of $\AdagK$ one sees there is an embedding $\AdagK \hookrightarrow \Robba$ via local expansions. 
Letting $\K((t))$ denote the usual ring
of one-way infinite Laurent series one may embed $\AK \hookrightarrow \K((t))$. Thus the entries of $N(t)$ and $F(t)$ may be
expanded locally around $t = 0$ to give elements in $\K((t))$ and $\Robba$ respectively.

\subsection{Limiting Frobenius structures}

\subsubsection{A surprising lemma}

Suppose that $\Delta(0) = 0$: then when one embeds entries in the matrix $F(t)$ into the Robba ring $\Robba$ one
expects infinitely many non-zero negative power terms in their Laurent series expansions. The following result is surprising
and rather deep.
 
\begin{lemma}\label{Lem-Surprise}
Let $N(t) = \sum_{i = -1}^\infty N_i t^i$ be an $r \times r$ matrix such that the entries of $tN$ are power series
over $\K$ convergent in the open unit disk, and $N_{-1}$ is a nilpotent matrix. Let $F(t) = \sum_{i = -\infty}^{\infty}
F_i t^i$ be an $r \times r$ matrix whose entries are (two-way infinite) Laurent series over $\K$ convergent
on some open annulus with outer radius $1$. Suppose $N$ and $F$ satisfy (\ref{Eqn-FN}). Then
$F_i = 0$ for $i < 0$, so $F(t)$ converges on the whole open unit disk.
\end{lemma}

\begin{proof}
This is exactly \cite[Lemma 6.5.7]{KKcubic}, which in turn is a consequence of \cite[Proposition 17.5.1]{KKbook}.
The proof has two steps. First, one assumes that $t N(t) = \N_0$, a nilpotent matrix. Then (\ref{Eqn-FN}) 
becomes 
\[ t\frac{dF}{dt} + \N_0 F = p F \sigma(\N_0).\]
Now compare coefficients of $t^i$ on both sides and observe that
the map on matrices $X \mapsto \N_0 X + X (i - p \sigma(\N_0))$ has all eigenvalues equal to $i$ since
$\N_0$ and $\sigma(\N_0)$ are nilpotent, by \cite[Lemma III.8.4]{DGS}. Hence we see $F_i = 0$ for $i \ne 0$.
Next, to reduce the general case to the case $t N(t) = \N_0$ one uses a deep ``transfer theorem'' of Christol
and Dwork \cite[Theorem 13.7.1]{KKbook}.
\end{proof}

\begin{note}\label{Note-Lift}
The choice of Frobenius lift $t \mapsto t^p$ is key to the proof --- this was stressed to the author by
Kiran Kedlaya. In the original deformation algorithm,
the author used the lift $s \mapsto s^p$ where $s$ is the local parameter around $t = 1$; that is,
$t \mapsto (t-1)^p + 1$. With this choice of lift one expects an essential singularity at $t = 0$ 
in $F(t)$ unless the family $\Xp/\Sp$ has a ``smooth extension'' over the origin. This simple
observation is the key to extending the deformation method to the case of degenerations --- it eluded
the author for many years.
\end{note}

\subsubsection{Definition and uniqueness}\label{Sec-LFSdefuni}

We now define our limiting Frobenius structures. Note that in this section all bases $\Basis_{rig}$ 
for $H^m_{rig}(\Xp/\Sp)$ arise by
extension of scalars from a basis $\Basis$ of $H^m_{dR}(\XK/\SK)$.

\begin{definition}\label{Def-LFS}
Let $F(t)$ and $N(t)$ be the matrices for the Frobenius and connection, respectively, w.r.t. some basis
$\Basis_{rig}$ of $H^m_{rig}(\Xp/\Sp)$. Assume that the hypotheses of Lemma \ref{Lem-Surprise} are satisfied.
Write $\N_0 := (tN(t))_{t=0}$ for the nilpotent
residue matrix, $\Fr_0 := F(0)$ and $\HH_0 := \langle \Basis_{rig} \rangle_{\K}$. Then by
equation (\ref{Eqn-FN}) we have $\N_0 \Fr_0 = p\Fr_0 \sigma(\N_0)$. We call $(\HH_0,\N_0,\Fr_0)$ the
limiting Frobenius structure (in dimension $m$ of the family $\XW \rightarrow \SW$) 
at $t = 0$ w.r.t. the basis $\Basis_{rig}$.
\end{definition}

\begin{lemma}
With notation and hypotheses as in Definition \ref{Def-LFS}, the matrix $\Fr_0$ is invertible.
\end{lemma}
 
\begin{proof}
Define $M := - N^t$ and $G := (F^{-1})^t$. Then the pair $(M,G)$ is the ``dual overconvergent $F$-isocrystal''
and satisfies the differential equation
\[ \frac{dG}{dt} + M(t)G(t) = G(t) \sigma(M(t))p t^{p-1}. \]
Now $tM(t)$ converges on the $p$-adic open unit disk and $M(t)$ has residue matrix
$-\N_0^t$, which is nilpotent. Hence by Lemma \ref{Lem-Surprise}, $G(t) = \sum_{i = 0}^\infty G_i t^i$
converges on the open unit disk. But $FG^t = I$ and so $F_0 G_0^t = I$ and $\Fr_0$ is invertible.  
\end{proof}
 
\begin{lemma}\label{Lem-LFSunique}
Assume there are two bases $\Basis_{rig}$ and $\Basis_{rig}^\prime$ which satisfy the hypothesis
in Definition \ref{Def-LFS}. Let $(\HH_0,\N_0,\Fr_0)$ and $(\HH_0^\prime,\N_0^\prime,\Fr_0^\prime)$ be the
associated limiting Frobenius structures at $t = 0$. Then there exists an invertible matrix $H_0$
over $\K$ such that $\N^\prime_0 = H_0 \N_0 H_0^{-1}$ and $\Fr_0^\prime = H_0 \Fr_0 \sigma(H_0)$.
\end{lemma}

\begin{proof}
Let $H(t)$ be the local expansion around the origin 
of the change of basis matrix from $\Basis$ to $\Basis^\prime$; hence also from $\Basis_{rig}$ to $\Basis_{rig}^\prime$.
Let $N(t)$ and $N(t)^\prime$ denote the connection matrices w.r.t. the two bases,
and $\N_0 = (tN(t))_{t=0}$,  $\N_0^\prime = (tN(t)^\prime)_{t=0}$ the two nilpotent matrices.
(Beware here and elsewhere that the prime superscript has nothing to do with differentiation.)
Then
$(tN^\prime) = H(tN)H^{-1} + t \frac{dH}{dt} H^{-1}$ and so $(tN^\prime) H - H (tN) = t \frac{dH}{dt}$. Write
$H(t) = H_{-s} t^{-s} + \cdots$ as a Laurent series with matrix coefficients where $H_{-s} \ne 0$. So 
$\N_0^\prime H_{-s} - H_{-s} \N_0 = -s H_{-s}$. But the operator on matrices $X \mapsto
\N_0^\prime X - X \N_0$ has eigenvalues the difference of the eigenvalues of the two
nilpotent matrices $\N_0$ and $\N_0^\prime$, by \cite[Lemma III.8.4]{DGS}. 
That is, its eigenvalues are all zero. So $s = 0$ and
$H$ has entries in $\K[[t]]$. By a symmetrical argument so does $H^{-1}$ and so $H(t) = 
H_0 + H_1 t + \cdots $ where $H_0$ is invertible. Thus $\N^\prime_0 = H_0 \N_0 H_0^{-1}$. Since the
Frobenius map $\Fr$ is $\sigma$-linear, the change of basis formula shows
$\Fr_0^\prime = H_0 \Fr_0 \sigma(H_0)$.
\end{proof}

Thus we may speak of ``the'' limiting Frobenius structure $(\HH_0,\N_0,\Fr_0)$ as an abstract triple of
a vector space $\HH_0$ over $\K$, a nilpotent linear operator $\N_0$ and an invertible $\sigma$-linear 
operator $\Fr_0$ such that
$\N_0 \Fr_0 = p \Fr_0 \N_0$, provided it exists w.r.t. at least one basis.

\subsubsection{Existence}

The limiting Frobenius structure might only exist after making a cover of the base. We now
prove the existence of a ``smallest'' limiting Frobenius structure under a lifting hypothesis.
 
For $e$ a positive integer, write $\AW_e :=  \W[s,1/\Delta(s^e)]$, $(\AK)_e := \K[s,1/\Delta(s^e)]$ and
$(\AdagK)_e$ for the weak completion of $\AW_e$ tensored by $\K$. 
Define $\AW \rightarrow \AW_e$, $\AK \rightarrow (\AK)_e$ and
$\AdagK \rightarrow (\AdagK)_e$ by $t \mapsto s^e$. Let $\SW_e := \Spec(\AW_e)$,
so there is a morphism $\SW_e \rightarrow \SW$.
One may pull-back the family $\XW \rightarrow \SW$ via $\SW_e \rightarrow \SW$
to get $\XW_e \rightarrow \SW_e$ where $\XW_e := \XW \times_{\SW} \SW_e$. Write 
$\Xp_e \rightarrow \Sp_e$ and $(\XK)_e \rightarrow (\SK)_e$ for the maps on special and
generic fibres. The morphism $\XW_e \rightarrow \SW_e$ is smooth (even when
$p$ divides $e$ and $\SW_e \rightarrow \SW$ is not) and we
have as before
\begin{equation}\label{Eqn-e}
 H^m_{rig}(\Xp_e/\Sp_e) \cong H^m_{dR}((\XK)_e/(\SK)_e) \otimes_{(\AK)_e} (\AdagK)_e.
\end{equation}

\begin{definition}\label{Def-LFS2}
Assume that the hypotheses of Definition \ref{Def-LFS} are satisfied for the family
$\XW_e \rightarrow \SW_e$ after pulling back the base $\SW_e \rightarrow \SW$ for some positive integer $e$.
Let $(\HH_0,\N_0,\Fr_0)$ be the limiting Frobenius structure for $\XW_e \rightarrow \SW_e$ at $s = 0$.
We call $(\HH_0,\N_0,\Fr_0,e)$ a limiting Frobenius structure for $\XW \rightarrow \SW$ at $t = 0$.
\end{definition}

Note that if one has a chain of covers $\SW_{e^\prime} \rightarrow \SW_e \rightarrow \SW$ and a limiting
Frobenius structure $(\HH_0,\N_0,\Fr_0,e)$ ``on'' $\SW_e$ one may pull it back to a limiting
Frobenius structure $(\HH_0,(e^\prime/e) \N_0,\Fr_0,e^\prime)$ ``on'' $\SW_{e^\prime}$. We then say that
the latter factors through the former.

We now state our main theorem on the existence of limiting Frobenius structures. Recall that Assumption (SP)
is still in place. We shall call the $p$-adic open unit disk around $t = 0$ with the origin removed the ``punctured'' $p$-adic open unit disk around $t = 0$.

\begin{theorem}\label{Thm-LFSexist}
Assume that the polynomial $\Delta(t)$ has no zeros in $\bar{\K}$ within the punctured $p$-adic open unit disk around
$t = 0$. Then a limiting Frobenius structure $(\HH_0,\N_0,\Fr_0,e)$ at $t = 0$ exists. Moreover, there
exists a unique ``smallest'' such limiting Frobenius structure, such that any other factors through it. 
\end{theorem}

\begin{proof}
Let $\Basis$ be any basis for $H^m_{dR}(\XK/\SK)$, $r$ be the size of this basis, and $N(t)$ the matrix for the Gauss-Manin connection
w.r.t. this basis. 
The matrix $N(t)$ converges on the punctured $p$-adic open unit disk around $t = 0$, since it has poles
only at the roots of $\Delta(t)$. Theorem \ref{Thm-SSR} guarantees the existence of a smallest positive
integer $e$ and change of basis matrix $H \in \GL_r (\K[s,s^{-1}])$ (here $t = s^e$) such that the matrix for the Gauss-Manin connection on $(\Basis_e)_{[H(s)]}$ has only a simple pole at $s = 0$ with nilpotent residue matrix. Here
$e$ is taken to be the lowest common multiple of the denominators of the eigenvalues $\lambda_i$ of the residue matrix
which occur in the proof of Theorem \ref{Thm-SSR}.
(Note that the hypothesis of computability on the base field $\F$ is only used to establish that the change of basis
matrix $H(s)$ is computable.) Denote this
connection matrix by $N(s)^\prime$. Since $H \in \GL_r (\K[s,s^{-1}])$, the matrix $N(s)^\prime$ converges on the punctured
$p$-adic open unit disk around $s = 0$. Thus the hypotheses of Definition \ref{Def-LFS2} apply after making
the base change $\SW_e \rightarrow \SW$ and we
may define a limiting Frobenius structure at $t = 0$.

A necessary condition for the existence of a limiting
Frobenius structure after pulling back the base $\SW_{e^\prime} \rightarrow \SW$ is that $e^\prime \lambda_i$ are
integers. Hence for any such $e^\prime$ we have that $e$ divides $e^\prime$, hence $\SW_{e^\prime} \rightarrow \SW$
factors through $\SW_e \rightarrow \SW$. Thus we can pull back the limiting Frobenius structure
$(\HH_0,\N_0,\Fr_0,e)$ via $\SW_{e^\prime} \rightarrow \SW_e$ and apply Lemma \ref{Lem-LFSunique} for the final claim.
\end{proof}

We call the ``smallest'' limiting Frobenius structure which occurs in Theorem \ref{Thm-LFSexist}, ``the'' limiting
Frobenius structure at $t = 0$.

\begin{note}\label{Note-pande}
Provided the characteristic $p$ does not divide $e$, then the morphism $\SW_e \rightarrow \SW$ reduces
to an unramified cover $\Sp_e \rightarrow \Sp$. Thus here we have made the overconvergent $F$-isocrystal
$H^m_{rig}(\Xp/\Sp)$ ``unipotent around zero'' after making an unramified cover of the base $\Sp$. So our construction
here agrees with one arising from the semistable reduction theorem. If $p$ divides $e$ then our construction
seems unnatural, since the cover $\Sp_e \rightarrow \Sp$ is totally ramified. 
The author suspects that the hypothesis 
on $\Delta(t)$ in Theorem  \ref{Thm-LFSexist} cannot be satisfied when the order $e$ of the finite part of the
local monodromy around the origin of the family $\XK \rightarrow \SK$ is divisible by $p$. He has a few examples which
suggest this, but no conceptual explanation as to why the order of the monodromy should be related to 
roots of $\Delta(t)$ within the punctured $p$-adic open unit disk.
\end{note}

Let us conclude with a definition which brings together the properties required for
Theorem \ref{Thm-LFSexist} to hold.

\begin{definition}\label{Def-GL}
We shall say that a morphism $\Xp \rightarrow \Sp \subseteq \Aff^1_{\sfq}$ of $\fq$-varieties has a good lift around zero if it arises by base
change from a morphism $\XW \rightarrow \SW$ of $\W$-schemes with the following property.
Assumption (SP) is satisfied (Definition \ref{Def-SP}), and writing $\SW = \Spec(\W[t,1/\Delta(t)])$ the polynomial $\Delta(t)$
has no zeros in $\bar{\K}$ within the punctured $p$-adic open unit disk around zero.
\end{definition}

So by Theorem \ref{Thm-LFSexist}, limiting Frobenius structures exist when one has a good lift around zero.

\subsubsection{Computability: definitions and a sketch}

We now turn to the question of computing limiting Frobenius structures. The field $\K$ is uncountable and
thus we immediately encounter two problems. First, one cannot perform exact arithmetic in $\K$. Second,
we cannot directly apply the results of Section \ref{Sec-Char0}, since there the base field
$\F$ is assumed to be countable. A convenient solution to both problems is to work in a suitable
algebraic number field $\LL \subset \K$, using exact arithmetic in $\LL$ when we wish to apply the ``algebraic'' theory from Section \ref{Sec-Char0}, and approximating elements in $\K$ by those in $\LL$ during the
``analytic'' parts of the algorithm.

Precisely, assume that we have an algebraic number field $\LL$ with ring of integers $\OL$ in which the prime $p$ is 
inert with residue field $\fq$. We have the natural embedding $\OL \hookrightarrow \W$ given by localisation at $p$. First, we give an ``algebraic'' version of Definition \ref{Def-GL}.

\begin{definition}\label{Def-GAL}
We shall say that a morphism $\Xp \rightarrow \Sp  \subseteq \Aff^1_{\sfq}$ of $\fq$-varieties has a good algebraic 
lift around zero if it arises by base change from a morphism $\XOL \rightarrow \SOL$ of $\OL$-schemes with the following property.
Assumption (SP) is satisfied for $\XW := \XOL \times_{\OL} \W \rightarrow \SW := \SOL \times_{\OL} \W$
(Definition \ref{Def-SP}), and writing $\SOL = \Spec(\OL[t,1/\Delta(t)])$ the polynomial $\Delta(t)$
has no zeros in $\bar{\K}$ within the punctured $p$-adic open unit disk around zero.
\end{definition}

When we have a good algebraic lift $\XOL \rightarrow \SOL$, 
write $\XL \rightarrow \SL$ for the base change by $\bullet \otimes_{\OL} \LL$.
By Theorem \ref{Thm-LFSexist} ``the'' limiting Frobenius structure $(\HH_0,\N_0,\Fr_0,e)$
exists for $\XW \rightarrow \SW$ when one has a good algebraic lift around zero. 
Unfortunately, we cannot hope to compute it,
since any matrix for $\Fr_0$ has entries in the uncountable field $\K$. Thus we need a notion of
an ``approximate'' limiting Frobenius structure.

For $N$ and $r$ positive integers and $\Fr_0 \in \GL_r(\K)$, we shall call a matrix $\tilde{\Fr}_0 \in \GL_r(\LL)$ a
$p^N$-approximation to $\Fr_0$ if $\ord_p(\tilde{\Fr}_0 - \Fr_0) \geq N$. 

\begin{definition}
Let $N$ be a positive integer. 
Assume that $\Xp \rightarrow \Sp$ has a good algebraic lift around zero.
A $p^N$-approximation to the limiting Frobenius structure at $t = 0$
in dimension $m$ for the smooth morphism $\XW \rightarrow \SW$ is a quadruple
$(\HH_0,\N_0,\tilde{\Fr}_0,e)$ with the following properties,
\begin{itemize}
\item{$\HH_0$ is a vector space over $\K$, $\N_0$ a nilpotent matrix over $\LL$, and
$\tilde{\Fr}_0$ an invertible matrix over $\LL$. }
\item{the limiting Frobenius structure for $\XW \rightarrow \SW$ in dimension $m$ at $t = 0$ exists and
w.r.t. some basis is $(\HH_0,\N_0,\Fr_0,e)$ with 
$\tilde{\Fr}_0$ a $p^N$-approximation to $\Fr_0$.}
\end{itemize}
\end{definition}

The ``algebraic'' part of the algorithm for computing approximate limiting Frobenius structures is
to find an integer $e$ and basis for $H^m_{dR}((\XL)_e/(\SL)_e)$ such that the matrix
for the connection has a simple pole at $s = 0$ with nilpotent residue matrix. 
(We use analogous notation for pulling back $\SL$-schemes via the morphism $t \mapsto s^e$
as before; that is, introduce a subscript $e$.) This allows us to compute exactly the nilpotent monodromy
matrix $\N_0$, and of course the integer $e$. The following theorem shows this can be accomplished
provided we know the connection matrix on some initial basis for $H^m_{dR}(\XL/\SL)$.

\begin{theorem}\label{Thm-NsGL}
There exists an explicit deterministic algorithm with the following input and output.
The input is a matrix for the Gauss-Manin connection
$\nabla$ w.r.t. some basis of $H^m_{dR}(\XL/\SL)$, where $\Xp \rightarrow \Sp$ has a good
algebraic lift around $t = 0$ . The output is a matrix $N(s)$ for $\nabla_e$ w.r.t. some
explicit basis
of $H^m_{dR}((\XL)_e/(\SL)_e)$ such that $N(s)$ converges on the punctured $p$-adic open unit disk
around $s = 0$ and has a simple pole at $s = 0$ with nilpotent residue matrix.
\end{theorem}

\begin{proof}
Let $\Basis$ be the basis used for the input. Then the connection matrix w.r.t. this basis has only poles
at the roots of $\Delta(t)$ and thus converges on the punctured $p$-adic open unit disk around $t = 0$.
 Now we apply
Theorem \ref{Thm-SSR} exactly as in the proof of Theorem \ref{Thm-LFSexist} noting that
$\LL$ is computable.
\end{proof}

The ``analytic'' part of the algorithm is to compute a suitable approximation to the matrix for Frobenius on this basis.
Let us first formalise our notion of approximate Frobenius matrices. For $e \geq 1$ define as before
$(\AdagK)_e$ to be the weak completion of $\W[s,1/\Delta(s^e)]$ tensored by $\K$. Write $(\AL)_e := \LL[s,1/\Delta(s^e)]$ and note that $(\AL)_e \hookrightarrow (\AdagK)_e$.
Let $N$ and $r$ be positive integers. We call $\tilde{F}(s) \in \GL_r((\AL)_e)$ a $p^N$-approximation to 
$F(s) \in \GL_r((\AdagK)_e)$ if $\ord_p(\tilde{F}(s) - F(s)) \geq N$. Thus our aim is to compute
a $p^N$-approximation $\tilde{F}(s)$ to the matrix $F(s)$ for Frobenius on this basis, and then ``specialise'' this
approximation at $s = 0$ to obtain a $p^N$-approximation $\tilde{\Fr}_0$ to $\Fr_0$. The approximation $\tilde{F}(s)$ can be found using the original deformation method. 

Rather than state a formal theorem on approximate computability of limiting Frobenius structures in this
general setting, let us focus now on the case of degenerations of projective hypersurfaces.

\subsection{Hypersurfaces}

In this section we prove our main result (Theorem \ref{Thm-LFShypersurfaces}) on the
computability of limiting Frobenius structures for degenerations of hypersurfaces, and
discuss the asymptotic complexity of the algorithm underlying this theorem (Section \ref{Sec-Complexity}).

\subsubsection{Computability of limiting Frobenius structures}\label{Sec-CLFSHyper}

Let us make a convenient definition to ease the statement of our main theorem. As before
$\OL$ is the ring of integers of an algebraic number field $\LL$ which is inert at the rational prime $p$ with
residue field $\fq$, and
$\OL \hookrightarrow \W$ is the embedding. Now take $\F$ to be the computable field $\LL$ and
define
notation in a similar manner to the first paragraph of Section \ref{Sec-CPFS} but taking $\PF_t$ to have coefficients
in $\OL[t]$ instead of $\LL[t]$.  Precisely, let $\Delta(t) \in \OL[t]$ be any polynomial which has as a factor the
polynomial $\Delta_{n,d}(\PF_t)$ defined by specialising
$\Delta_{n,d}$ at the coefficients of $\PF_t$. Assume $\Delta(t) \ne 0$. Let
$\SOL := \Spec(\OL[t,1/\Delta(t)])$ and $\XOL \subset \pr^{n+1}_{\SOL}$ be the scheme defined by
$\PF_t = 0$.  Denote by $\Xp \rightarrow \Sp$ (respectively $\XL \rightarrow \SL$) the morphism
of $\fq$-varieties (respectively $\LL$-varieties) obtained by tensoring 
$\XOL \rightarrow \SOL$ by $\bullet \otimes_{\OL} \fq$ (respectively $\bullet \otimes_{\OL} \LL$).
So each fibre
in the family $\XL \rightarrow \SL$ is smooth and in general position (Definition \ref{Def-GenPos}).
Now make the additional assumption that $\Delta(1) \ne 0 \bmod{p}$ and $\Delta(t)$ has
no zeros in the punctured $p$-adic open unit disk around $t = 0$.

\begin{definition}\label{Def-GALhyper}
We shall say that a smooth family of hypersurfaces $\Xp \rightarrow \Sp$ of degree $d$ over $\fq$ is in general position and
has a good algebraic lift
around $t = 0$  if
it arises by reduction modulo $p$ from a morphism $\XOL \rightarrow \SOL$ of $\OL$-schemes of the type
described in the paragraph immediately above.
\end{definition}

\begin{theorem}\label{Thm-LFShypersurfaces}
There exists an explicit deterministic algorithm which takes as input a positive integer $N$ and a smooth family of hypersurfaces
$\Xp \rightarrow \Sp$ over $\fq$ which is in general position and has a good algebraic lift around $t = 0$, and gives as output a $p^N$-approximation to the primitive middle-dimensional limiting Frobenius structure 
at $t = 0$ of the family $\Xp \rightarrow \Sp$.
\end{theorem}

\begin{proof}
Let $\Basis := \Basis_{\XL}$ be the
Dwork basis for the primitive middle-dimensional cohomology of $\XL \rightarrow \SL$. 
Theorem \ref{Thm-GMDwork} shows that one may compute the matrix $N(t)$ for the Gauss-Manin connection on
$\Basis$. Theorem \ref{Thm-NsGL} then allows us to compute the pull-back $(\SOL)_e \rightarrow \SOL$ and
basis change matrix $H(s)$ needed so that the Picard-Fuchs matrix $N(s)^\prime$, say, 
w.r.t. the basis $(\Basis_e)_{[H(s)]}$ has a simple pole at $s = 0$ with nilpotent residue matrix and
converges on the punctured $p$-adic open unit disk. Let $-\delta$ be the minimum of the $p$-adic
orders of $H(s)$ and $H(s)^{-1}$. Suppose that one can compute a $p^{N+2\delta}$-approximation
$\tilde{F}(t)$ to the matrix for Frobenius $F(t)$ on $\Basis_{rig}$. Then by substituting $t = s^e$ and
changing basis by $H(s)$ one obtains a $p^N$-approximation $\tilde{F}(s)^\prime$, say, to the matrix
for Frobenius $F(s)^\prime$ on $(\Basis_{rig,e})_{[H(s)]}$.  By Lemma \ref{Lem-Surprise} the matrix
$F(s)^\prime$ has no pole at $s = 0$ and specialising $sN(s)^\prime$ and $\tilde{F}(s)^\prime$ at $s = 0$ gives
the data of a $p^N$-approximation of the limiting Frobenius structure. (One has to take some
care in ``specialising'' $\tilde{F}(s)^\prime$ at $s = 0$ as this matrix may have a pole. To obtain an approximation of
$F(0)^\prime$ embed the entries of $\tilde{F}(s)^\prime$ in $\LL((t))$ and take the coefficient of $t^0$.) 

To compute the matrix $\tilde{F}(t)$ itself we use the deformation method.
This requires two inputs. First, the Picard-Fuchs matrix $N(t)$. Second, the matrix for the Frobenius action on the basis $\Basis_1$ of $H^n_{rig}(\Xp_1)$ obtained from $\Basis_{rig}$ via $H^n_{rig}(\Xp_1) \cong
H^n_{rig}(\Xp/\Sp) \otimes_{\AdagK} \K$ $(t \mapsto 1)$ \cite[Theorem 3.2]{L1}. The matrix for the Frobenius action on the fibre
is needed to some higher $p$-adic precision, due to precision loss during the deformation method.
Precisions required can be calculated as in \cite[Section 6]{KKcubic},
\cite[Section 5]{L1}. One uses Kedlaya's algorithm to calculate the Frobenius matrix at the fibre $t = 1$
\cite[Section 3.5]{AKR}.
\end{proof} 

\begin{note}\label{Note-ThmLFShypersurfaces}
Note that here Assumption (SP) is satisfied, since the primitive cohomology is just that of the family of complements
of the smooth hypersurfaces in $\pr^{n+1}_\W$. Of course, by primitive middle-dimensional limiting Frobenius
structure we mean the limiting Frobenius structure in the middle dimension for this family of complements. Indeed, there is a small clash of notation between Section \ref{Sec-Char0}
and \ref{Sec-Charp} --- we really wish to take the family ``$\XW \rightarrow \SW$'' in Section \ref{Sec-Charp} to be the family of complements
of hypersurfaces, not the family ``$\XL \rightarrow \SL$'' of hypersurfaces itself.
\end{note}

\subsubsection{Complexity}\label{Sec-Complexity}

We do not attempt an asymptotic analysis of the algorithm in Theorem \ref{Thm-LFShypersurfaces}, but
confine ourselves instead to a few comments.

Assume that the matrix for the Gauss-Manin connection w.r.t. to the Dwork basis has a simple pole with nilpotent
residue matrix. Then our algorithm is exactly the same as the original deformation method, except that
we use the Frobenius lift $t \mapsto t^p$ rather than $t - 1 \mapsto (t - 1)^p$ (Note \ref{Note-Lift}). When
we solve the Picard-Fuchs differential system locally around $t = 1$ during the algorithm, using the lift
$t  \mapsto t^p$ is slower in practice than using the lift $(t - 1) \mapsto (t-1)^p$ (cf. \cite[Section 5.1]{L1});
however, provided one uses the ``Tsuzuki method'' \cite[Section 5.1.2]{L1} the asymptotic complexity is the same.
To obtain asymptotic complexity bounds, it is simpler to assume that one only computes the matrix for the
connection w.r.t. the Dwork basis to a suitable $p$-adic precision rather than as an exact matrix of rational
functions over an algebraic number field. This is all worked out explicitly in the analogous setting of ``Dwork cohomology'' in  \cite[Sections 5, 6.1, 10: Step 2]{Lfocm}. By comparison with the analysis in \cite[Section 10]{L1}, one sees that using an
approximate connection matrix our algorithm certainly has running time which
is polynomial in $p,\log_p(q),(d+1)^n$ and $N$.

A necessary condition for our assumption on the connection matrix to be met is that the pull-back
$e = 1$, and hence the Picard-Lefschetz transformation is unipotent, cf. Theorem \ref{Thm-LMT}. When this
geometric condition is met (for example, a semistable degeneration of curves) then the author has observed
that frequently in practice our assumption on the connection matrix holds true. However, it is certainly
not the case that it is always true, and the author does not understand its true geometric significance, if any.

When the matrix for the connection does not have a simple pole with nilpotent residue matrix w.r.t. the Dwork
basis, then the author makes no claims on the complexity of the algorithm. Computing pull-backs and
using shearing transformations to prepare eigenvalues is not too problematic --- both were usually
necessary for the examples in Section \ref{Sec-Exs}.
The most time consuming
problem in practice is when there is not even a simple pole, and one needs to use a cyclic vector; see
Example \ref{Ex-Quintic}.

\subsection{Limit Frobenius-Hodge structures}\label{Sec-LFHS}

We conclude our description of degenerations in positive characteristic with a less formal discussion of the 
integral structure and Hodge filtration on our limiting Frobenius structures. Our hope is to convince
the reader without detailed proofs that these additional structures exist and can be computed in
quite wide generality.

Assume that we have an algebraic number field $\LL$ with ring of integers $\OL$, 
and a smooth and proper morphism $\XOL \rightarrow \SOL$ of $\OL$-schemes, where
$\SOL = \Spec(\AOL)$ with $\AOL := \OL[t,1/\Delta(t)]$ for some $\Delta(t) \in \OL[t]$.
We denote by $\XL \rightarrow \SL$ extension of scalars of $\XOL \rightarrow \SOL$ from $\OL$ to $\LL$.
We follow our usual convention of using the subscript $e$ to denote pulling-back along the
base via $s^e := t$.
Theorem \ref{Thm-SSR} guarantees the existence of $e \geq 1$ and a basis
for $H^m_{dR}((\XL)_e/(\SL)_e)$ such that the matrix for the Gauss-Manin connection
w.r.t. this basis has a simple pole at $s = 0$ with nilpotent residue matrix. To simplify notation
in our discussion let us suppose that $e = 1$, and write $\Basis$ for the basis. 

Let $p$ be a rational prime which is inert in $\OL$, and $\XW \rightarrow \SW$ be the
morphism obtained by base change from $\OL$ to $\W$. Denote by
$H^m_{cris}(\XW/\SW)$ relative crystalline cohomology.
Write $\AW := \AOL \otimes_{\OL} \W$, let $\hat{\AW}$ be the $p$-adic completion of $\AW$, and
let $\Adag$ be the weak completion of $\AW$ (not tensored by $\K$), see Section \ref{Sec-RRC}.
We claim then that $\Basis$ maps naturally to a free $\hat{A}$-basis $\Basis_{cris}$ for $H^m_{cris}(\XW/\SW)$ provided $p$ is sufficiently large. We argue as follows.

For almost all such rational primes $p$ one may canonically map $\Basis$ to a
subset $\Basis_{dR}$ of $H^m_{dR}(\XW/\SW)$ and provided $p$ is sufficiently large (depending on the geometry
of the morphism $\XL \rightarrow \SL$) then $H^m_{dR}(\XW/\SW)$ will be free with
$\Basis_{dR}$ as an $\AW$-basis. One expects a crystalline relative comparison theorem,
$H^m_{cris}(\XW/\SW) \cong H^m_{dR}(\XW/\SW) \otimes_{\AF} \hat{\AF}$, although the author
does not know of a reference. Thus extending scalars $\Basis_{dR}$ maps to an $\hat{A}$-basis
$\Basis_{cris}$ of $H^m_{cris}(\XW/\SW)$. One can make this construction of $\Basis_{cris}$ entirely explicit for
families of smooth hypersurfaces using a relative version of the theorems and algorithms in \cite{AKR}.

Write $\HH_\W := H^m_{cris}(\XW/\SW)$ and denote by $\Fr$ the Frobenius on $\HH_\W$.
The cohomology group $\HH_\W$ inherits a filtration $\HH_\W^\bullet$ from
the Hodge filtration on $H^m_{dR}(\XW/\SW)$.
Applying Mazur's ``Frobenius-Hodge theorem'' \cite{BM} fibre-by-fibre it follows that $\Fr$
 is ``divisible by $p^i$'' on $\HH_\W^i$; that is, $\Fr(\HH_\W^i) \subseteq p^i \HH_\W$.
 
Assume now that $p$ is such that $\Basis_{cris}$ is a basis for $\HH_\W$.
Let $F(t)_{cris}$ denote the matrix for $\Fr$ w.r.t. the basis $\Basis_{cris}$. Since
 $\Basis_{cris}$ is also naturally a basis for $H^m_{rig}(\Xp/\Sp)$ the matrix $F(t)_{cris}$ has
entries in $\AdagK \cap \hat{A} = \Adag$.
The connection matrix $N(t)_{cris}$ has coefficients in $\AOL \subset \AW$ and
a simple pole at $t = 0$ with nilpotent residue matrix. For all
but finitely many primes $N(t)_{cris}$ will converge on the punctured $p$-adic open unit disk
around zero. Thus if we further exclude this finite set of primes, the basis $\Basis_{cris}$ will satisfy the hypotheses
of Definition \ref{Def-LFS} and one can specialise at $t = 0$ to obtain a triple
$(\HH_0,\N_0,\Fr_0)_\W$ of a free $\W$-module $\HH_0$ with operators $\N_0$ and $\Fr_0$ such that
$\N_0 \Fr_0 = p \Fr_0 \N_0$ and $\N_0^{m+1} = 0$. Moreover, there is a natural filtration
$\HH_0^{\bullet}$ such that $\Fr_0(\HH_0^i) \subseteq p^i \HH_0$ coming from the Hodge filtration
on $\HH_\W$. 

Allowing once again $e > 1$ we have constructed the ``limiting Frobenius-Hodge structure'' $(\HH_0^\bullet,\N_0,\Fr_0,e)_\W$ from Section \ref{Sec-Intro}. Note that when $\XL \rightarrow \SL$ ``extends'' to a semistable
degeneration over the origin (cf. Definition \ref{Def-SSex})
we have $e = 1$, and the pair ``$(\HH_0^\bullet,\N_0)_{\W} \otimes_{\W} \C$'' is exactly the limiting
mixed Hodge structure from complex algebraic geometry \cite[Chapter 11]{PS}.

\begin{note}\label{Note-Names}
The term ``limiting mixed Hodge structure'' 
arises naturally since there is already the notion of a ``mixed Hodge structure'' and a ``Hodge structure''.  In rigid
or crystalline cohomology there is no formal notion of a ``Frobenius structure'' to the author's knowledge, although
the term is used loosely, e.g. in ``differential equation with a Frobenius structure''. The object which we call a
``limiting Frobenius structure'' might be better called a ``limiting mixed $F$-isocrystal'', and
our ``limiting Frobenius-Hodge structure'' a ``limiting mixed Hodge $F$-crystal''.
\end{note}

\section{Degenerations and ``Clemens-Schmidt''}\label{Sec-CS}

This section contains no original theorems, but rather gives a conjectural interpretation of the geometric
significance of the limiting Frobenius structures we define. The notation in Sections \ref{Sec-CAV},
\ref{Sec-SSR} and \ref{Sec-LMEC} mainly follows the original sources, and only in Section \ref{Sec-CSRC}
do we return to the notation from Section \ref{Sec-Charp}.

\subsection{Complex algebraic varieties}\label{Sec-CAV}

The material in this section is based upon the exposition in \cite{DM}. Throughout
$H^m(\bullet)$ and $H_m(\bullet)$ denote singular cohomology and homology, respectively, with rational coefficients

Let $\Delta$ denote the complex open unit disk. A degeneration is a proper flat
holomorphic map $\pi: \XC \rightarrow \Delta$ of relative dimension $n$ such that
$\XC_t := \pi^{-1}(t)$ is a smooth complex variety for $t \ne 0$, and $\XC$ is
a Kahler manifold. A degeneration is semistable if the central fibre $\XC_0$ is a
reduced divisor with global normal crossings; that is, writing $\XC_0 = \sum_i \XC_{0,i}$ as
a sum of irreducible components, each $\XC_{0,i}$ is smooth and the $\XC_{0,i}$ meet
transversally. 

We recall a fundamental theorem on degenerations, due to D. Mumford with the assistance of F. Knudsen and
A. Waterman, see \cite{DM} or \cite[Page 102]{DSM}.

\begin{theorem}[Semistable Reduction]
Given a degeneration $\pi: \XC \rightarrow \Delta$ there exists a base
change $b: \Delta \rightarrow \Delta$ (defined by $t \mapsto t^e$ for some
positive integer $e$), a semistable degeneration $\psi: \YC \rightarrow
\Delta$ and a diagram
\[
\begin{array}{ccccc}
\YC & \stackrel{f}{\dashrightarrow} & \XC_b & \longrightarrow & \XC\\
       &  \stackrel{\psi}{\searrow} & \downarrow &  & \downarrow \\
       &                 & \Delta & \stackrel{b}{\longrightarrow} & \Delta
\end{array}
\]
such that $f: \YC \dashrightarrow \XC_b$ is a bimeromorphic map obtained by blowing up and blowing
down subvarieties of the central fibre. 
\end{theorem}

Let $\pi: \XC \rightarrow \Delta$ be a degeneration, and $\pi^* : \XC^* \rightarrow
\Delta^*$ the restriction to the punctured disk $\Delta^* := \Delta \backslash \{0\}$. Fix a smooth
fibre $\XC_t$ ($t \ne 0$). Since $\pi^*$ is a $C^\infty$ fibration, the fundamental group $\pi_1(\Delta^*)$ acts on the
cohomology $H^m(\XC_t)$, for each $0 \leq m \leq 2n$. The map 
\[ T_m: H^m(\XC_t) \rightarrow H^m(\XC_t) \]
induced by the canonical generator of $\pi_1 (\Delta^*)$ is called the {\it Picard-Lefschetz} or
{\it local monodromy} transformation. 

We have the following theorem which is usually attributed independently to P. Griffiths, A. Grothendieck and A. Landman.
(See \cite[Page 106]{DSM} for all but the claim on $k$, and \cite[Theorem I$^\prime$]{AL} for the claim on $k$.)

\begin{theorem}[Local Monodromy Theory]\label{Thm-LMT}
The local monodromy transformation $T_m$ is quasi-unipotent, with index of unipotency at most $m$. In other words, there is some
$k$ such that 
\[ (T_m^k - I)^{m+1} = 0.\]
Moreover, if the central fibre $\XC_0$ is a (possibly non-reduced) divisor $\sum_i k_i \XC_{0,i}$ with
the $\XC_{0,i}$ smooth hypersurfaces on $\XC$ meeting transversally, then one may take
$k = \lcm \{k_i\}$. In particular, if $\pi: \XC \rightarrow \Delta$ is semistable then $T$ is unipotent ($k = 1$).
\end {theorem}

Note that by blowing up subvarieties of the central fibre, one may always replace the original degeneration
with one in which the central fibre is a (possibly non-reduced) union of smooth hypersurfaces crossing transversally without
changing the family outside of the central fibre. 

Let $\pi : \XC \rightarrow \Delta$ be a semistable degeneration of fibre dimension $n$, and $T_m$ the local
monodromy transformation on $H^m(\XC_t)$ for some $0 \leq m \leq 2n$. 
Since here the monodromy is unipotent, we can define the
logarithm $H^m(\XC_t) \stackrel{N_m}{\rightarrow} H^m(\XC_t)$ by the finite sum
\[ N_m  :=  \log (T_m) = (T_m - I) - \frac{1}{2} (T_m - I)^2 + \frac{1}{3} (T_m - I)^3 - \cdots.\]
The homology and cohomology groups
with rational coefficients of the central fibre $\XC_0$ and the chosen fibre $\XC_t$ ($t \ne 0$) carry canonical mixed
Hodge structures \cite[Pages 114-116]{DM}. These fit together in an elaborate exact sequence. 

\begin{theorem}[Clemens-Schmidt sequence]\label{Thm-CS}
Let $\pi :\XC \rightarrow \Delta$ be a semistable degeneration and $N_m$ the logarithm of local monodromy on
$H^m(\XC_t)$ for some $t \ne 0$.
There exists an exact sequence of mixed Hodge structures
\[ \cdots \rightarrow H_{2n + 2 -m} (\XC_0) \rightarrow H^m (\XC_0) \rightarrow H^m (\XC_t) \stackrel{N_m}{\rightarrow} H^m(\XC_t) \]
\[ \rightarrow
H_{2n-m}(\XC_0) \rightarrow H^{m+2}(\XC_0) \rightarrow H^{m+2} (\XC_t) \stackrel{N_{m+2}}{\rightarrow} H^{m+2}(\XC_t)
\rightarrow \cdots.\]
\end{theorem}

All the maps in this exact sequence are described explicitly on \cite[Pages 108-109]{DM}. 

\subsection{Varieties over finite fields}

The material in Sections \ref{Sec-SSR} and \ref{Sec-LMEC} follows closely the
presentation in \cite{Illusie}. Beware that the ``$K$'' in these sections need not be the $p$-adic
field from Section \ref{Sec-Charp}. Indeed, we are interested in geometric rather than arithmetic degenerations, and
so in our application $K$ is equicharacteristic.

\subsubsection{Semistable reduction}\label{Sec-SSR}

Let $R$ be a henselian discrete valuation ring with field of fractions $K$ and
residue field $k$ of characteristic $p > 0$. Choose an algebraic closure $\bar{K}$ of
$K$, and denote by $\bar{R}$ the normalisation of $R$ in $\bar{K}$, and
$\bar{k}$ the residue field of $\bar{R}$ (which is an algebraic closure of $k$).

Let $S := \Spec(R)$, $s := \Spec(k)$, $\bar{s} := \Spec(\bar{k})$,
$\eta := \Spec(K)$, $\bar{\eta} := \Spec(\bar{K})$, and $f: X \rightarrow S$ be
a proper morphism. We say that $f$ is semistable if
locally for the \'{e}tale topology $X$ is $S$-isomorphic to 
$S[t_1,\cdots,t_n]/(t_1 \dots t_r - \pi)$ where $\pi$ is a uniformizer on $R$. This
implies that the generic fibre $X_\eta$ is smooth, $X$ is regular, and
the special fibre $X_s$ is a divisor with normal crossings on $X$.

The following fundamental conjecture is known to be true at least when $X_\eta$ is a smooth projective curve.

\begin{conjecture}[Semistable reduction]\label{Conj-GSS}
If $X_\eta$ is smooth over $\eta$ then there exists a finite extension $\eta^\prime$
of $\eta$ such that $X_{\eta^\prime}$ admits a proper and semistable model $X^\prime \rightarrow S^\prime$ on the
normalisation $S^\prime$ of $S$ in $\eta^\prime$.
\end{conjecture}

\subsubsection{Local monodromy in \'{e}tale cohomology}\label{Sec-LMEC}

Let $I$ be the inertia group, defined by the exact sequence
\[ 1 \rightarrow I \rightarrow \Gal(\bar{K}/K) \rightarrow \Gal(\bar{k}/k) \rightarrow 1.\]
Write $G := \Gal(\bar{K}/K)$ and let $\rho: G \rightarrow \GL (V)$ be
a continuous representation of $G$ to a finite-dimensional $\Q_{\ell}$-vector space ($\ell \ne p$).
We say that the representation is quasi-unipotent if there is an open
subgroup $I_1$ of $I$ such that the restriction to $I_1$ of $\rho$ is unipotent, i.e.
$\rho(g)$ is unipotent for all $g \in I_1$.

Let $f: X \rightarrow S$ be a proper morphism, and $H^m(X_{\bar{\eta}},\Q_{\ell})$ denote
\'{e}tale cohomology with coefficients in $\Q_{\ell}$ of $X_{\bar{\eta}}$. Then there exists a 
local monodromy representation

\begin{equation}\label{Eqn-rho}
 \rho_m: G \rightarrow \GL\, H^m(X_{\bar{\eta}},\Q_{\ell}).
 \end{equation}

We have following following analogues of Theorem \ref{Thm-LMT}.

\begin{theorem}\cite[Corollaire 3.3]{Illusie}
The local monodromy representation on \linebreak $H^m(X_{\bar{\eta}},\Q_{\ell})$ for a semistable
degeneration $f : X \rightarrow S$ is unipotent. More precisely, for $g \in I$ we have
$(\rho_m(g) - 1)^{m+1} = 0$.
\end{theorem}

\begin{corollary}\cite[Corollaire 3.4]{Illusie}\label{Cor-GLMT}
Let $f: X \rightarrow S$ be a proper morphism with $X_{\eta}$ smooth over $\eta$.
Assuming geometric semistable reduction (Conjecture \ref{Conj-GSS}), the local monodromy
representation on $H^m(X_{\bar{\eta}},\Q_{\ell})$ is quasi-unipotent.
\end{corollary}

\begin{note}
Grothendieck proved that any $\ell$-adic representation of $G$ is quasi-unipotent, under a mild
hypothesis on $k$ \cite[Th\'{e}or\`{e}me 1.2]{Illusie}. In the case in which $k$ is a finite field, this gives a proof of Corollary
\ref{Cor-GLMT} which does not need geometric semistable reduction.
\end{note}

Let $f: X \rightarrow S$ be a semistable degeneration and $\rho_m$ the associated monodromy
representation (\ref{Eqn-rho}).
There exists a homomorphism $t_{\ell}: I \rightarrow \Z_{\ell}(1)$ and a unique nilpotent morphism
\[ N_m: H^m(X_{\bar{\eta}},\Q_{\ell})(1) \rightarrow H^m(X_{\bar{\eta}},\Q_{\ell})\]
such that
\[ \rho_m(g) = \exp(N_m t_{\ell}(g)) \mbox{ for all $g \in I$}.\]
Assume now that $k$ is finite, and $S$ is the henselisation of a closed point on a smooth curve over $k$.
(These conditions ensure in particular that the weight-monodromy conjecture holds true, see \cite[Pages 41-42]{Illusie}.)
The following is the analogue in $\ell$-adic \'{e}tale cohomology of one piece of the Clemens-Schmidt sequence.

\begin{theorem}[Local invariant cycle theorem]\cite[Corollaire 3.11]{Illusie}
For all $m$ there exists an exact sequence
\[ H^m(X_{\bar{s}},\Q_{\ell}) \rightarrow H^m(X_{\bar{\eta}},\Q_{\ell}) \stackrel{N_m}{\rightarrow} H^m(X_{\bar{\eta}},\Q_{\ell})(-1).\]
\end{theorem}

\subsubsection{Rigid cohomology}\label{Sec-CSRC}

There is not such a coherently worked out picture for the rigid (or crystalline) cohomology of a degeneration.
See \cite{Bruno,Mieda} for some partial results for rigid cohomology (and the references therein for crystalline
cohomology). By analogy though the following suggests itself.

Let $\Xp \rightarrow \Sp \subseteq \Aff^1_{\sfq}$ 
be a morphism of varieties over $\fq$ which has a good lift around $t = 0$
(Definition \ref{Def-GL}). Then the limiting Frobenius structure $(\HH_0,\N_0,\Fr_0,e)$  for
$\XW \rightarrow \SW$ in dimension $m$ around $t = 0$ exists and is unique (Theorem
\ref{Thm-LFSexist} and Lemma \ref{Lem-LFSunique}). We shall write $(\HH_0^m,\N_0,\Fr_0,e_m)$
when we wish to indicate the dependence of the limiting Frobenius structure on $m$.

Recall here that $\SW = \Spec(\W[t,1/\Delta(t)])$ for
some polynomial $\Delta(t)$ which has no zeros in $\bar{\K}$ within the
punctured $p$-adic open unit disk around zero but may vanish at
$t = 0$. 
It is interesting to note that if $\Delta(0) = 0 \bmod{p}$ (that is, $0 \not \in \Sp$), then
we must have $\Delta(0) = 0$. (This follows by looking at Newton polygons since $\Delta(t)$ is not identically
zero modulo $p$ (for otherwise $\Sp$ is empty) and therefore $\Delta(t)$ contains a root in the $p$-adic open unit disk.)

\begin{definition}\label{Def-SSex}
Let $\Xp \rightarrow \Sp$ be any smooth morphism of 
varieties over $\fq$ where $\Sp \subseteq \Aff^1_{\sfq} \backslash \{0\}$. Let $\Tp:=\Sp \cup \{0\}$ be the variety
obtained by adjoining the origin to $\Sp$. We shall say that $\Xp \rightarrow \Sp$ extends to a semistable degeneration
over the origin if there exists a semistable degeneration $\Yp \rightarrow \Tp$ which restricts
to $\Xp \rightarrow \Sp$ over $\Sp$.
\end{definition}

Here by ``semistable degeneration'' we mean in the sense of Section \ref{Sec-SSR} when one localises 
$\Yp \rightarrow \Tp$ around the origin. In this situation we shall write $\Xp_0$ for the fibre of
$\Yp \rightarrow \Tp$ at the origin and refer to it as the degenerate fibre. Note $\Xp \rightarrow \Sp$ may
extend to a semistable degeneration over the origin in more than one way; for example, one can blow
up smooth points on the degenerate fibre. Let us also define a semistable $\fq$-variety as one which is
\'{e}tale locally isomorphic to $\fq[t_1,\cdots,t_n]/(t_1 \cdots t_r)$ for some $n$ and $r \leq n$; so the
degenerate fibre in a semistable degeneration is semistable.

We recall the definition of a basic object in rigid cohomology.

\begin{definition}
An $F$-isocrystal $M$ over $\K$ is a finite-dimensional $\K$-vector space equipped with an injective
$\sigma$-linear map $F$. For $m \in \Z$, define $M(m)$ to be the $F$-isocrystal with the same underlying
vector space as $M$ but with the
action of $F$ multiplied by $p^{-m}$.
\end{definition} 

The $F$-isocrystal which is associated to the limiting Frobenius structure $(\HH_0^m,\N_0,\Fr_0,e_m)$ is the pair 
$(\HH_0^m,\Fr_0)$.

Define $e := \lcm_{0 \leq m \leq 2n}(e_m)$ where $n$ is the fibre dimension of $\Xp \rightarrow \Sp$.
 One may pull back the base $\XW_{e} \rightarrow \XW_{e_m}$ and the ``smallest''
limiting Frobenius structure $(\HH_0^m,\N_0,\Fr_0,e_m)$ in dimension $m$ is replaced by $(\HH_0^m,(e/e_m)\N_0,\Fr_0,e)$. Note that this has no effect on the kernel and cokernel of the monodromy operator, since
$\N_0$ is multiplied by a non-zero integer, i.e., there is no ambiguity in replacing $e_m$ by $e$.

\begin{conjecture}\label{Conj-CS}
Let $\Xp \rightarrow \Sp \subseteq \Aff^1_{\sfq} \backslash \{0\}$ 
be a morphism of varieties over $\fq$ which has a good lift around $t = 0$ and
denote by $(\HH_0^m,\N_0,\Fr_0,e_m)$ the limiting Frobenius structure at $t = 0$ in dimension $m$. Define
$e := \lcm_{0 \leq m \leq 2n}(e_m)$.
Then the morphism $\Xp_e \rightarrow \Sp_e$ obtained by pulling back the base $\Sp_e \rightarrow \Sp$ extends
to a semistable degeneration over the origin with degenerate fibre $\Xp_{e,0}$. Moreover, for any such extension 
there exists a ``Clemens-Schmidt'' exact sequence of $F$-isocrystals
\[ \cdots \rightarrow H_{2n + 2 -m,rig}(\Xp_{e,0})(-n-1) \rightarrow 
H^m_{rig}(\Xp_{e,0}) \rightarrow \HH_0^m  \stackrel{\N_0}{\rightarrow} \HH_0^m(-1)  \]
\[ \rightarrow H_{2n  -m,rig}(\Xp_{e,0})(-n-1) \rightarrow 
H^{m+2}_{rig}(\Xp_{e,0}) \rightarrow \HH_0^{m+2}  \stackrel{\N_0}{\rightarrow} \HH_0^{m+2}(-1) \rightarrow \cdots  \]
for some covariant ``rigid homology'' functor $H_{j,rig}(\bullet)$ on semistable $\fq$-varieties which vanishes
outside of dimensions $0 \leq j \leq 2n$.
\end{conjecture}

The twistings here are chosen to be consistent with those on the filtered vector spaces \cite[Page 108-109]{DM}, and
of course the equation $\N_0 \Fr_0 = p \Fr_0 \N_0$.

For $n = 1$ and $m$ odd our ``Clemens-Schmidt'' sequence is
\[ 0 \rightarrow H^1_{rig}(\Xp_{e,0}) \rightarrow \HH_0^1 \stackrel{\N_0}{\rightarrow} \HH_0^1(-1) \rightarrow 
H_{1,rig}(\Xp_{e,0})(-2) \rightarrow 0.  \]
Thus according to Conjecture \ref{Conj-CS} the polynomial
\[ \det(1 - T \Fr_0^{\log_p(q)} | \Ker(\N_0)) \]
is the numerator of the zeta function of the semistable curve $\Xp_{e,0}$. 
(The author is implicitly assuming here the existence of a trace formula for
semistable curves, cf. (\ref{Eqn-TraceFormula}).)
One can test this for smooth
plane curves by combining our algorithm for computing limiting Frobenius structures of hypersurfaces with
explicit ad hoc computations of ``stable limits'' of plane curves.

For $n > 1$ our ``Clemens-Schmidt'' sequence would involve in a non-trivial manner the unknown
``rigid homology'' groups $H_{j,rig}(\Xp_{e,0})$. It seems reasonable to conjecture though that
the degenerate fibre $\Xp_{e,0}$ arises by reduction modulo $p$ from a $\W$-scheme $\XW_{e,0}$
and dropping the Frobenius structure we have an isomorphism of $\K$-vector spaces
\[ H_{j,rig}(\Xp_{e,0}) \cong H_j ((\XW_{e,0})_{\K}). \]
where $(\XW_{e,0})_{\K}$ is the generic fibre and $H_j(\bullet)$ singular homology with coefficients in $\K$.
Thus under this further conjecture the dimensions at least of our unknown spaces are amenable to computation.

\section{Examples}\label{Sec-Exs}

In this section we present some explicit examples of degenerations of curves, surfaces and
three-folds with the aim of illustrating our construction and providing evidence in support of
Conjecture \ref{Conj-CS}. We follow here the paradigm from experimental science rather than mathematics.
That is, we test experimentally consequences of our conjecture and if there is no inconsistency declare
we have found evidence in support of it.

\subsection{Preparation}

We first discuss our choice of examples, and explain exactly what we computed in each case.

\subsubsection{The families}

In our examples the family $\Xp \rightarrow \Sp$ is defined by an equation
\begin{equation}\label{Eqn-Family}
 P_t := (1-t)P_0 + t P_1 = 0
\end{equation}
for specific choices of homogeneous polynomials $P_0,\,P_1 
\in \fp[x_0,\cdots,x_{n+1}]$. We write the variables
$x_0,x_1,\cdots,x_{n+1}$ as $X,Y,Z$ and so on. The good algebraic lift
$\XW_{\sZ} \rightarrow \SW_{\sZ}$ (Definition \ref{Def-GALhyper})
of our family is defined by the same equation
only with the polynomials $P_0$ and $P_1$ thought of as having integer coefficients.
(See the first paragraph of Section \ref{Sec-CLFSHyper} for the exact construction of 
$\XW_{\sZ} \rightarrow \SW_{\sZ}$ from the polynomial $P_t$; here we take $\Delta(t)$ to be
$\Delta_{n,d}(P_t)$.)

Our algorithm outputs a suitable approximation to the limiting Frobenius structure $(\HH_0,\N_0,\Fr_0,e)$ at $t = 0$
for the primitive middle dimensional cohomology of $\Xp \rightarrow \Sp$, cf. Theorem \ref{Thm-LFShypersurfaces}. By suitable approximation
we mean here a $p^N$-approximation for $N$ sufficiently large in each example for the author to
identify a plausible integer polynomial $\det(1 - T \Fr_0|\HH_0)$. We do not claim this polynomial is provably correct; however, our choice of $N$ is based upon expectations on the size of the coefficients from
the weight-monodromy conjecture (Section \ref{Sec-WMC}) and consideration of the $\Z_p$-lattice structure and
Hodge filtration on the limiting Frobenius structure (Section \ref{Sec-LFHS} and \cite[Section 9.3.2]{L1}).
 
\subsubsection{The weight-monodromy conjecture}\label{Sec-WMC}

The nilpotent monodromy operator $\N_0$ defines
an increasing filtration 
\[ \{0\} =: (\HH_0)_{-1} \subseteq (\HH_0)_0 \subseteq (\HH_0)_1 \subseteq \cdots \subseteq (\HH_0)_{2n} := \HH_0\]
on the $\Qp$-vector space $\HH_0$ \cite[Pages 106-107]{DSM}. 
The filtration is defined as a ``convolution'' of the two filtrations
$(\Ker \N_0^{k})_{k = 0}^{n+1}$ and $(\Ima \N_0^i)_{i = 0}^{n+1}$ \cite[(1.5.5)]{Illusie}. 
Since $\N_0 \Fr_0 = p \Fr_0 \N_0$ the filtration is therefore stabilised by the Frobenius $\Fr_0$.
In the case of curves ($n = 1$) we have that $\N_0^2 = 0$ and the filtration is
\begin{equation}\label{Eqn-CurveFiltration}
 \{0\} \subseteq \Ima(\N_0) \subseteq \Ker(\N_0) \subseteq \HH_0.
\end{equation}
For surfaces ($n = 2$) we have $\N_0^3 = 0$ and here the filtration is
\begin{equation}\label{Eqn-SurfaceFiltration}
 \{0\} \subseteq \Ima(\N_0^2) \subseteq \Ima(\N_0) \cap \Ker(\N_0) \subseteq
\Ima(\N_0) + \Ker(\N_0) \subseteq \Ker(\N_0^2) \subseteq \HH_0. 
\end{equation}
Note that for surfaces when $\N_0^2 = 0$ then (\ref{Eqn-SurfaceFiltration}) becomes
(\ref{Eqn-CurveFiltration}), only with indices shifted by one. In any dimension $n$ we have
$\N_0^{n+1} = 0$ (Theorem \ref{Thm-LMT}), but when in fact $\N_0 = 0$ then the filtration becomes
\begin{equation}\label{Eqn-NzeroFiltration}
(\HH_0)_{n-1} := \{0\} \subseteq \HH_0 =: (\HH_0)_{n}.
\end{equation}
These cases cover the ones we shall encounter in our examples.

For $j$ such that $2j \in \Z$ 
and $Q(T) \in 1 + T \Q[T]$ we shall call $Q(T)$ a weight $j$ Weil polynomial w.r.t. $p$ if all of its reciprocal roots have complex absolute value $p^j$. We shall say that the weight-monodromy conjecture holds for our construction if for
each $0 \leq j \leq 2n$, we have $\det(1 - T \Fr_0| (\HH_0)_{j}/(\HH_0)_{j-1})$ a weight $j/2$ Weil polynomial w.r.t.
$p$.

For each of our examples we state the dimensions of the subspaces in the monodromy filtration, and 
we verified that the weight-monodromy conjecture holds for our construction. 
This gives evidence for the
first part of Conjecture \ref{Conj-CS}, in the following slightly loose sense. 
If the morphism $\Xp_e \rightarrow \Sp_e$ does extend
to a semistable degeneration over the origin, then the weight-monodromy conjecture holds in the setting
of $\ell$-adic \'{e}tale cohomology \cite[Pages 41-42]{Illusie}. So one at least expects that it should hold
for our construction, although the author does not offer a proof of this.

\subsubsection{Stable limits of curves}

The subspace $(\HH_0)_{2n-1}$
in the filtration for curves ($n = 1$) is $\Ker \N_0$. We write down $Q(T) := \det(1 - T \Fr_0 | \Ker \N_0)$ 
for our examples of curve degenerations. 
According to Conjecture \ref{Conj-CS} after pulling back $s^e := t$ our degeneration should extend to a semistable
degeneration over the origin with degenerate fibre the curve $\Xp_{e,0}$ and $Q(T)$ should be the numerator of the
zeta function of  $\Xp_{e,0}$. We provide evidence in support of this in Section \ref{Sec-ConjCSevidence}, 
using the recipes in \cite{HM} to
construct at least geometrically the ``stable limits'' of our families at the origin.
Writing $\Xp_0$
for the fibre at $t = 0$ in our original family, defined by the polynomial $P_0$, it is amusing to
note that $Q(T)$ can have nothing whatsoever to do with the zeta function of the 
singular fibre $\Xp_0$ itself.

\begin{note}
In \cite{RK} Kloosterman gives an interesting discussion of
the application of both the deformation method and Kedlaya's algorithm to singular projective hypersurfaces.
Several low degree examples in which the cohomology is pure ``algebraic'' are presented. It is not clear
to the present author that the methods based on deformation in that paper can be generally applied; for example,
the lift of Frobenius chosen may cause difficulties, cf. Note \ref{Note-Lift}. However, the idea that some appropriate
version of the deformation method should compute the stable limit of a degeneration of curves is due to Kloosterman.
\end{note}

\subsubsection{Comments on implementation}

Our choice of examples was limited in several ways. First, it was convenient to assume that
the degree $d$ divides $p-1$ and take $P_1 := x_0^d + \cdots + x_{n+1}^d$, so that we could
easily write down the Frobenius matrix at $t = 1$ cf. \cite[Section 1.5]{KKcubic}. Second, we restricted
attention mainly to examples in which the Picard-Fuchs matrix w.r.t. the Dwork basis had only a simple
pole at $t = 0$ (and converged in the punctured $p$-adic open unit disk). This avoided having to
use a cyclic vector, which we found time consuming. See Example \ref{Ex-Quintic} for an exception.
Third, for the quartic surface examples we chose polynomials $P_0$ with few non-zero terms.
This was a needed for the Gr\"{o}bner basis computation to finish within a reasonable time (see
Note \ref{Note-Pancratz} for a method avoiding this). 
With these restrictions we were still able to examine many different types of degenerations (some of
which were suggested to the author by Damiano Testa).

The experiments were carried out using a machine with four quad-core processors ($2.60$ GHz) and $64$ GB
RAM. 
Approximate timings were: Example \ref{Ex-DoubleConic} ($1$ minute), Example \ref{Ex-3cusps} ($3$ minutes),
Example \ref{Ex-QuinticTP} ($19$ hours), Example \ref{Ex-Sextic} ($4$ hours), Example \ref{Ex-SexticDamiano} ($44$ hours), Example \ref{Ex-Quintic} ($48$ minutes), Example \ref{Ex-QuarticA1} ($41$ minutes), Example \ref{Ex-QuarticA3}
($1$ minute), Example \ref{Ex-QuarticA2} ($95$ minutes), Example \ref{Ex-Roman} ($2$ hours), Example 
\ref{Ex-QuarticCone} ($4$ hours), Example \ref{Ex-Quartic1D} ($11$ hours)
and Example \ref{Ex-3fold} ($6$ hours). Space consumption varied accordingly up to around $6$ GB RAM.

\subsection{The experiments}

We now present the results of our computational experiments.

\subsubsection{Examples directly supporting Conjecture \ref{Conj-CS}}\label{Sec-ConjCSevidence}

In the following examples we were able to compute geometrically the stable limit of the degeneration
and provide strong evidence supporting Conjecture \ref{Conj-CS}.

\begin{example}\label{Ex-DoubleConic} 
Consider the degeneration (\ref{Eqn-Family}) of a diagonal quartic curve to a double conic defined
by taking $P_1  :=  X^4 + Y^4 + Z^4$ and
\[
\begin{array}{rcl}
P_0 & := & (X^2 + Y^2 + 3Z^2 + 3XY + YZ + 2XZ)^2.
\end{array}
\]
Here $e = 2$ and by pulling back the base via $s^2 := t$ we get monodromy
operator $\N_0 = 0$. The monodromy filtration (\ref{Eqn-CurveFiltration}) has spaces of dimensions $0,0,6,6$, respectively.
For the prime $p := 5$ the reverse characteristic polynomial of the Frobenius
$\Fr_0$ on $\Ker(\N_0) = \HH_0$ is
\[
Q(T)  =  1 - 6T + 23T^2 -58T^3 + 115T^4 - 150T^5 + 125T^6.
\]
This is a weight $\frac{1}{2}$ Weil polynomial w.r.t. $5$. The zeta function of the fibre $\Xp_0$ is
$1/(1-T)(1-5T)$ which is not related to the polynomial $Q(T)$. However, the ``stable limit''
$\Xp_{2,0}$ of the
family around zero is a hyperelliptic curve of genus $2$ which is constructed as a double-cover of the conic 
\[  C := X^2 + Y^2 + 3Z^2 + 3XY + YZ + 2XZ = 0 \]
ramified at the eight points of intersection of $C = 0$ with $P_1 = 0$ \cite[Pages 133-134]{HM}.
By direct computation of these points and projection to an axis 
one finds that over $\f_5$ the curve $\Xp_{2,0}$ has equation
\[ y^2 = a (x^8 + 3x^7 + 2x^6 + 4x^3 + x^2 + 2x + 1) \]
for some non-zero $a \in \f_5$. Taking $a$ to be a non-square in $\f_5$ gives a curve
whose zeta function is indeed $Q(T)/(1-T)(1-5T)$. (For $a$ a square the zeta function is
$Q(-T)/(1-T)(1-5T)$.) So this computation is consistent with Conjecture \ref{Conj-CS}.
\end{example}

Observe that if we changed the polynomial $P_1$ in Example \ref{Ex-DoubleConic} then the stable limit
of the degeneration would (usually) be different, since it depends upon the points of intersection of $P_1$ with
the conic. Thus the stable limit can depend upon the family, not just the degenerate fibre.

\begin{example}\label{Ex-3cusps}
Consider the degeneration (\ref{Eqn-Family}) of a diagonal quartic curve to a $3$-cuspidal quartic 
(cusps $(0:0:1),(0:1:0),(1:0:0)$) defined by
taking $P_1  :=  X^4 + Y^4 + Z^4$ and
\[
\begin{array}{rcl}
P_0 & := & X^2Y^2 + Y^2Z^2 + Z^2X^2 - 2XYZ(X + Y + Z).
\end{array}
\]
Here $e = 6$ and pulling back the base via $s^6 := t$ the local monodromy becomes trivial; that is,
$\N_0 = 0$. So as before the monodromy filtration  (\ref{Eqn-CurveFiltration}) consists of spaces of dimensions $0,0,6,6$, respectively.
For the prime $p := 13$ the reverse characteristic polynomial of the Frobenius
$\Fr_0$ on $\Ker(\N_0) = \HH_0$ is
\[ Q(T)  =  (1 - 5T + 13T^2)^3. \]
This is a weight $\frac{1}{2}$ Weil polynomial w.r.t. $13$. The fibre $\Xp_0$ has geometric 
genus zero and zeta function $1/(1-T)(1-13T)$. The ``stable limit'' 
around zero for the family over $\C$ is a projective line $\pr^1_{\sC}$ and three non-intersecting copies
of an elliptic curve $\E_{\sC}$ with $j$-invariant zero each meeting the projective line transversely at one point;
cf. \cite[Pages 122-128]{HM}, where the elaborate calculation of the stable limit of a curve acquiring a cusp
is performed. The polynomial $1 - 5T + 13T^2$ is the numerator of the zeta function of an elliptic
curve $E := \E \times_{\sZ} \f_{13}$ where $\E_{\sC} := \E \times_{\sZ} \C$ has $j$-invariant zero, for
$\E$ some elliptic curve ``over $\Z$'' (indeed several). Thus this is consistent with
Conjecture \ref{Conj-CS} taking $\Xp_{6,0}$ to be $\pr_{\sff_{13}}^1 \cup E \cup E \cup E$ with appropriate
intersections.
\end{example}

\begin{example}\label{Ex-QuinticTP}
Consider the degeneration (\ref{Eqn-Family}) of a diagonal quintic to a quintic with a triple point $(0:0:1)$ defined
by taking $P_1  :=  X^5 + Y^5 + Z^5$ and
\[
\begin{array}{rcl}
P_0 & := & X^5 + Y^5 + (X^4 + 3X^3Y  + Y^4)Z + (2X^3 + XY^2 + 3Y^3)Z^2.
\end{array}
\]
Here $e = 3$ and pulling back the base via $s^3 := t$ gives non-trivial nilpotent monodromy $\N_0$. 
We have $\N_0^2 = 0$ and the monodromy filtration  (\ref{Eqn-CurveFiltration}) consists of spaces of dimensions $0,2,10,12$, respectively.
For $p :=11$ the reverse characteristic polynomial of the Frobenius acting $\Fr_0$
acting on $\Ker(\N_0)$ is
\[ Q(T) = (1-T)(1+T)(1 + 11 T^2)(1 + 5T + 22T^2 + 62T^3 + 242T^4 + 605T^5 + 1331T^6).\]
The latter degree 8 factor is a weight $\frac{1}{2}$ Weil polynomial w.r.t. $11$.
The ``stable limit'' around zero for the family over $\C$ is the normalisation of the
degenerate fibre meeting an elliptic curve with $j$-invariant zero at the three points
lying over the triple point of the degenerate fibre \cite[Pages 131-132]{HM}. The degree six
factor in $Q(T)$ is the numerator of the zeta function of the normalisation, and 
$1 + 11T^2$ the numerator of the zeta function of
$E := \E \times_{\sZ} \f_{11}$ where $\E_{\sC} := \E \times_{\sZ} \C$ has $j$-invariant zero, for
$\E$ some elliptic curve ``over $\Z$''. So this example supports
Conjecture \ref{Conj-CS}.
\end{example}

\begin{example}\label{Ex-Sextic}
Consider the degeneration (\ref{Eqn-Family}) of a diagonal sextic to a sextic curve with an ordinary double point
$(0:0:1)$ defined by taking $P_1 :=  X^6 + Y^6 + Z^6$ and
\[
\begin{array}{rcl}
P_0 & := & X^6 + Y^6 + 3Y^4Z^2 + 2X^3Z^3 + (X^2 + Y^2)Z^4.
\end{array}
\]
Here $e = 1$, and no pullback on the base was required, but the monodromy operator
$\N_0$ is non-zero.  The monodromy filtration  (\ref{Eqn-CurveFiltration}) consists of
spaces of dimensions $0,1, 19,20$, respectively.
For the prime $p :=7$ the 
reverse characteristic polynomial of the Frobenius
$\Fr_0$ acting on $\Ker(\N_0)$ is
\[
Q(T)  =  (1+T)(1 - T + 7T^2)(1 + T - 9T^3 + 49T^5 + 343T^6)\]
\[ \times (1 - 4T +19T^2 - 60T^3 + 179T^4 - 522T^5 + 1253T^6 - 2940T^7 + 6517T^8 \]
\[ - 9604T^9 + 16807T^{10}).
\]
The fibre $\Xp_0$ is semistable, and since the total space is smooth we have
a semistable degeneration.  Indeed $\Xp_{1,0} := \Xp_0$ has zeta function
$Q(T)/(1-T)(1-7T)$ as predicted by Conjecture \ref{Conj-CS}.
\end{example}

\begin{note}\label{Note-BadPrimes}
One can illustrate in a simple manner how our method can be used to compute the local factors
at ``bad primes'' in the L-series of a curve using Example \ref{Ex-Sextic}. Consider the 
plane curve over $\Z$ defined by the polynomial $R_0 := P_0 + 7 Z^6$ with $P_0$ the
polynomial from Example \ref{Ex-Sextic} but 
thought of as having integer coefficients. Then this curve is smooth over
$\Q$ but the reduction modulo the prime $p := 7$ is the singular sextic $\Xp_0$. The local
factor at $p := 7$ in the L-series of the curve defined by $R_0$ is $Q(T)$. Note here that one
computes this local factor by lifting $R_0  \bmod{7}$ to the integer polynomial
$P_0$, not the integer polynomial $R_0$ itself. The reason for this is that the family over
$\Q$ defined by the equation $(1-t)R_0 + t P_1$ is not smooth in the punctured $p$-adic open unit disk
around $t = 0$: since $R_0$ defines a smooth curve over $\Q$ but $R_0 \bmod{7}$ defines
a singular curve over $\f_7$ this forces there to be a singular fibre in the family for some
$t = t_0$ with $\ord_7(t_0) > 0$ (in fact, $\ord_7(t_0) = 1$ here).

Although this example looks artificial, it is somewhat typical. Let $R_0 \in \Z[X,Y,Z]$ be homogeneous
of degree $d$ defining a smooth curve over $\Q$, but a singular curve modulo some
prime $p$. One expects this singular curve to have a single rational node, which after a
linear change of variable is $(0:0:1)$. So after this change of variable $R_0 = P_0 + p(aX + bY)Z^{d-1} + pcZ^d$
for some $a,b,c \in \Z$ and $P_0(X,Y,Z)$ homogeneous of degree $d$ with no term in $Z^{d-1}$ or $Z^d$.
One expects the curve defined by the polynomial $P_0$ to have again just the single node $(0:0:1)$.
One then computes the local factor at $p$ using the lift $P_0$ of $R_0 \bmod{p}$ as in Example
\ref{Ex-Sextic}.
\end{note}

\begin{example}\label{Ex-SexticDamiano}
Consider the smooth curve of genus $4$ in $\pr^3$ defined as the intersection
\[
\begin{array}{rcl}
T^2 + X^2 + Y^2 + Z^2 & = & 0\\
T^3 + 2TXZ + X^3 - X^2Y + Y^3 + Z^3 & = & 0.
\end{array}
\]
Eliminating $T$ we obtain a birational curve in $\pr^2$ of degree $6$ with $6$ nodes, defined
by the polynomial
\[ P_0 := X^6 - X^5Y - 2X^5Z + 2X^4Y^2 + \tfrac{7}{2}X^4Z^2 + X^3Y^3 - 4X^3Y^2Z - 
    3X^3Z^3 \]
    \[ + \tfrac{1}{2}X^2Y^4 + 5X^2Y^2Z^2 - X^2YZ^3 + \tfrac{7}{2}X^2Z^4 - 
    2XY^4Z - 4XY^2Z^3 - 2XZ^5 \]
    \[ + Y^6 + \tfrac{3}{2}Y^4Z^2 + Y^3Z^3 + 
    \tfrac{3}{2}Y^2Z^4 + Z^6.\]
Consider now the degeneration (\ref{Eqn-Family}) to this six-nodal sextic defined by taking $P_1 := X^6 + Y^6 + Z^6$. Here
$e = 1$ and no pull-back on the base gives non-trivial nilpotent monodromy $\N_0$.
We have $\N_0^2 = 0$ and the monodromy filtration consists  (\ref{Eqn-CurveFiltration}) consists
of spaces of dimensions $0,6, 14,20$, respectively.
For $p := 7$  the reverse characteristic polynomial of the Frobenius $\Fr_0$
acting on $\Ker(\N_0)$ is
\[ Q(T) := (1 - T^2 + T^4 - T^6) \]
 \[ \times (1 + 2T + 11T^2 + 40T^3 + 72T^4 + 280T^5 + 539T^6 + 686T^7 + 2401T^8).\]
The fibre $\Xp_0$ is semistable with smooth total space, so we have a semistable degeneration. Indeed $\Xp_{1,0} := \Xp_0$ has zeta function 
$Q(T)/(1-T)(1-7T)$, which supports Conjecture \ref{Conj-CS}. The zeta function of the original genus $4$ curve in $\pr^3_{\sff_7}$ has numerator
the degree $8$ weight $\frac{1}{2}$ w.r.t. $7$ factor of $Q(T)$.
\end{example}

\subsubsection{Examples illustrating our construction}

The next examples are primarily illustrative, although in each case we verified the weight-monodromy
conjecture held for our construction (Section \ref{Sec-WMC}). It would be interesting to compute
a ``semistable limit'' for the degeneration in each case. Note that for the surfaces considered, cyclotomic factors in 
$Q(T/p)$ should correspond to algebraic curves on the ``semistable limit'', by the philosophy of the
Tate conjecture if not the conjecture itself.

\begin{example}\label{Ex-Quintic}
Consider the degeneration (\ref{Eqn-Family}) of a diagonal quintic to a non-reduced union of lines defined by
taking $P_1 :=  X^5 + Y^5 + Z^5$ and
\[
\begin{array}{rcl}
P_0 & := & XYZ^3.
\end{array}
\]
Here $e = 3$ and pulling back the base via $s^3 := t$ gives non-trivial nilpotent monodromy $\N_0$. 
We have $\N_0^2 = 0$ and the monodromy filtration  (\ref{Eqn-CurveFiltration}) consists of
spaces of dimensions $0,1,11,12$, respectively.
Note that in this example the connection matrix had a double-pole at $t = 0$ w.r.t. the initial Dwork
basis (see Corollary \ref{Cor-DworkBasis}) 
and so a cyclic vector was used to obtain a new basis so that the connection matrix
had only a simple pole (cf. Theorem \ref{Thm-ExplicitReg}).
For the prime $p := 31$ the reverse characteristic polynomial of the Frobenius
$\Fr_0$ acting on $\Ker(\N_0)$ is
\[ Q(T)  =  (1-T)(1 + 4T + 31T^2)(1 + 8T + 33T^2 + 248T^3 + 961T^4)^2. \]
The zeta function of the original singular fibre $\Xp_0$ is $1/(1-31 T)^3$. 
\end{example}

\begin{example}\label{Ex-QuarticA1}
Consider the degeneration (\ref{Eqn-Family}) of a diagonal quartic surface to a quartic surface with a pair of ordinary double points
($(\pm i: 1 : 0 : 0)$ where $i^2 = -1$) defined by taking $P_1  :=  X^4 + Y^4 + Z^4 + W^4$ and
\[
\begin{array}{rcl}
P_0 & := & X^4 + Y^4 + Z^4 + W^4 + 2X^2Y^2 + 2XYZW + 2ZW^3.
\end{array}
\]
Here $e = 2$ and pulling back the base via $s^2 := t$ gives trivial monodromy; that is, $\N_0 = 0$. 
The monodromy filtration (\ref{Eqn-SurfaceFiltration}) has subspaces of dimension $0,0,0,21,21,21$.
For the
prime $p := 5$ the reverse characteristic polynomial $Q(T)$ of the Frobenius $\Fr_0$
on $\HH_0$ satisfies
\[
Q(T/5)  =  (1-T)^7(1+T)^6(1+T^2)^2(1+ T + T^2)(1 + \tfrac{2}{5} T + T^2).
\]
The polynomial $Q(T)$ is a weight $1$ Weil polynomial w.r.t. $5$.  By naive point counting the
author verified that $\bmod\,T^8$ the local expansion at $T = 0$ of the zeta function of the singular fibre $\Xp_0$ is congruent to
$(1 + 5 T)^2$ times the local expansion of $1/(1-T)(1-5T)Q(T)(1-25T)$. This suggests that a semistable
degeneration can be constructed in which the degenerate fibre $\Xp_{2,0}$ is $\Xp_0$ blown up at
the two singular points.
\end{example}

\begin{example}\label{Ex-QuarticA3}
Consider the degeneration (\ref{Eqn-Family}) of a diagonal quartic to a quartic surface with a pair of $A_3$ singularities
($(0:0:0:1)$ and $(-1:0:0:1)$) defined by taking $P_1  :=  X^4 + Y^4 + Z^4 + W^4$ and
\[
\begin{array}{rcl}
P_0 & := & X^4 + Y^4 + 2Z^4 + W^2(X^2 + Y^2) + 2WX^3.
\end{array}
\]
Here $e = 4$ and pulling back the base via $s^4 := t$ gives trivial monodromy; that is, $\N_0 = 0$.
The monodromy filtration  (\ref{Eqn-SurfaceFiltration}) has subspaces of dimension $0,0,0,21,21,21$.
For the prime $p := 5$ the reverse characteristic polynomial $Q(T)$ of the Frobenius $\Fr_0$
on $\HH_0$
satisfies
\[
Q(T/5) =  (1-T)^5(1+T)^8(1 - \tfrac{8}{5}T + T^2)(1 + T^2)^3.
\]
The polynomial $Q(T)$ is a weight $1$ Weil polynomial w.r.t. $5$.
\end{example}

\begin{example}\label{Ex-QuarticA2}
Consider the degeneration (\ref{Eqn-Family}) of a diagonal quartic to a quartic surface with an $A_2$ singularity 
$(0:0:0:1)$ defined by $P_1  :=  X^4 + Y^4 + Z^4 + W^4$ and
\[
\begin{array}{rcl}
P_0 & := & X^4 + 2Y^4 + 2Z^4 + W^2(X^2 + Y^2) + 2WZ^3.
\end{array}
\]
Here $e = 3$ and pulling back the base via $s^3 := t$ gives trivial monodromy.
The monodromy filtration  (\ref{Eqn-SurfaceFiltration}) has subspaces of dimension $0,0,0,21,21,21$.
For the prime $p := 13$ the reverse characteristic polynomial $Q(T)$ of the Frobenius $\Fr_0$
on $\HH_0$ satisfies
\[ Q(T/13)  =  (1-T)^2(1+T)^3(1- T + T^2)(1+T+T^2)^2(1-T^2 + T^4)\]
\[ \times (1 - \tfrac{2}{13}T + \tfrac{16}{13}T^2 - \tfrac{6}{13}T^3 + \tfrac{16}{13}T^4
- \tfrac{2}{13}T^5 + T^6).\]
The polynomial $Q(T)$ is a weight $1$ Weil polynomial w.r.t. $13$.
\end{example}

\begin{example}\label{Ex-Roman}
Consider the degeneration (\ref{Eqn-Family}) of a diagonal quartic to the ``Roman surface'' defined
by taking $P_1  :=  X^4 + Y^4 + Z^4 + W^4$ and 
\[
\begin{array}{rcl}
P_0 & := & X^2Y^2 + X^2Z^2 + Y^2Z^2 + 2XYZW.
\end{array}
\]
Here $e = 2$ and pulling back the base via $s^2 := t$ gives non-trivial nilpotent monodromy; the monodromy
matrix $\N_0$ is such that $\N_0^2 \ne 0$ but $\N_0^3 = 0$.
The monodromy filtration (\ref{Eqn-SurfaceFiltration}) has subspaces 
of dimension $0,1,1,20,20,21$.
For the prime $p := 13$ the reverse characteristic polynomial $Q(T)$ of the Frobenius $\Fr_0$
on $\HH_0$ is
\[ Q(T) = (1 - 169 T)(1 - 13T)^7(1 + 13T)^{12} (1-T).\]
Note that $\Ker(\N_0)$ is not a subspace in the monodromy filtration, and
$\det(1 - T \Fr_0 | \Ker(\N_0)) = (1-13T)^6(1+13T)^{12}(1-T)$.
\end{example}

\begin{example}\label{Ex-QuarticCone}
Consider the degeneration (\ref{Eqn-Family}) of a diagonal quartic to a cone over a smooth quartic curve
defined by taking $P_1  :=  X^4 + Y^4 + Z^4 + W^4$ and
\[
\begin{array}{rcl}
P_0 & := & X^4 + 3X^3Y + XYZ^2 + Y^4 + 2Y^3Z + Z^4.
\end{array}
\]
Here $e = 4$ and pulling back the base via $s^4 := t$ gives trivial monodromy. 
The monodromy filtration  (\ref{Eqn-SurfaceFiltration}) has subspaces of dimension $0,0,0,21,21,21$.
For the prime
$p := 13$ the reverse characteristic polynomial $Q(T)$ of the Frobenius $\Fr_0$ on $\HH_0$
is such that
\[
Q(T/13) = (1-T)^3(1+T)^2(1+T^2)(1 + \tfrac{16}{13}T + \tfrac{21}{13}T^2 + \tfrac{16}{13}T^3 + 
\tfrac{15}{13}T^4
+ \tfrac{16}{13}T^5 + \tfrac{19}{13}T^6\]
\[ + \tfrac{24}{13}T^7 + \tfrac{19}{13}T^8 + \tfrac{16}{13}T^9
+ \tfrac{15}{13}T^{10} + \tfrac{16}{13}T^{11} + \tfrac{21}{13}T^{12} + \tfrac{16}{13}T^{13} + T^{14}).
\] 
The polynomial $Q(T)$ is a weight $1$ Weil polynomial w.r.t. $13$. The zeta function of the fibre
$\Xp_0$ is $R(13T)/(1-T)(1-13T)(1-169T)$ where $R(T) := 1 + 2T + 11T^2 - 24T^3 + 143T^4 + 338T^5 + 13^3T^6$ is 
the numerator of the zeta function of the smooth quartic curve. 
\end{example}

\begin{example}\label{Ex-Quartic1D}
Consider the degeneration (\ref{Eqn-Family}) of a diagonal quartic to an irreducible quartic surface with a one-dimensional
singular locus defined by taking $P_1  :=  X^4 + Y^4 + Z^4 + W^4$ and
\[
\begin{array}{rcl}
P_0 & := & X^2Y^2 + Y^2W^2 + X^2Z^2 + Y^2W^2 + X^2W^2 + Y^2Z^2 + 2XYZW.
\end{array}
\]
Here $e = 2$ and pulling back the base via $s^2 := t$ gives non-trivial monodromy matrix
$\N_0$ with $\N_0^2 = 0$. The monodromy filtration (\ref{Eqn-SurfaceFiltration}) has subspaces of dimension
$0,0,2,19,21,21$. For the prime $p := 13$ the reverse characteristic polynomial $Q(T)$ of the Frobenius $\Fr_0$ on $\HH_0$ is
\[ Q(T) = (1 - 13T)^7 (1 + 13T)^{10}(1 + 2T + 13T^2)(1 + 26T + 2197T^2).\]
The last two factors here are weight $\frac{1}{2}$ and weight $\frac{3}{2}$ Weil polynomials w.r.t. $13$, respectively.
\end{example}

\begin{example}\label{Ex-3fold}
Consider the degeneration (\ref{Eqn-Family}) of a diagonal cubic $3$-fold to a singular $3$-fold containing a pencil of
planes defined by taking $P_1  :=  X^3 + Y^3 + Z^3 + U^3 + V^3$ and
\[
\begin{array}{rcl}
P_0 & := & ZX^2 + UXY + VY^2.
\end{array}
\]
Here $e = 2$ and pulling back the base via $s^2 := t$ gives trivial monodromy. So $\N_0 = 0$ and
the monodromy filtration (\ref{Eqn-NzeroFiltration}) has subspaces of dimension $0,0,0,0,10,10,10,10$.
For the prime $p := 19$ the reverse characteristic polynomial $Q(T)$ of the Frobenius $\Fr_0$
on $\HH_0$ is such that
\[
Q(T/19)  =  (1-2T + 19T^2)^2(1 + 8T + 19T^2)^2(1-T + 19T^2).
\]
Thus $Q(T)$ is a weight $\frac{3}{2}$ Weil polynomial w.r.t. $19$.
\end{example}

\end{document}